\documentclass[11pt,a4paper,reqno,oneside]{amsart}
\usepackage{tabls}
\usepackage{url}
\usepackage[linktocpage = true]{hyperref}
\hypersetup{
 colorlinks = true,
 citecolor = blue,
}
\usepackage{color}
\definecolor{red}{rgb}{1,0,0}
\definecolor{green}{rgb}{0,1,0}
\definecolor{blue}{rgb}{0,0,1}
\definecolor{refkey}{gray}{.625}
\definecolor{labelkey}{gray}{.625}
\usepackage[lite]{amsrefs}
\usepackage[a4paper]{geometry}
\usepackage{amsfonts}
\usepackage{amssymb}
\usepackage{mathrsfs}
\usepackage{amsmath}
\usepackage{amsthm}
\usepackage{tikz}
\usepackage{tikz-cd}
\usepackage{amscd}
\usepackage{enumerate}
\usepackage{paralist}
\usepackage{graphicx}
\usepackage{mathtools}
\usepackage{kantlipsum}
\usepackage{stmaryrd}
\usepackage{extarrows}
\allowdisplaybreaks

\makeatletter
 \def\title@font{\normalsize\bfseries}
 \let\ltx@maketitle\@maketitle
 \def\@maketitle{\bgroup%
 \let\ltx@title\@title%
 \def\@\title{\resizebox{\textwidth}{!}{%
  \mbox{\title@font\ltx@title}%
 }}%
 \ltx@maketitle%
 \egroup}
\makeatother

\theoremstyle{plain}
\newtheorem*{zorn*}{Zorn's lemma}
\newtheorem*{tychonoff*}{Tychonoff's theorem}

\newtheorem{definition}{Definition}[section]
\newtheorem{lem}[definition]{Lemma}
\newtheorem{Cor}[definition]{Corollary}
\newtheorem{Thm}[definition]{Theorem}
\newtheorem*{theorem*}{Theorem}
\newtheorem{Def}[definition]{Definition}
\newtheorem{def-prop}[definition]{Definition-Proposition}
\newtheorem{prop}[definition]{Proposition}
\newtheorem{prop-def}[definition]{Proposition-Definition}
\newtheorem{Ex}[definition]{Example}
\newtheorem{Rem}[definition]{Remark}

\DeclareMathOperator{\id}{id}

\DeclareMathOperator{\ad}{ad}

\DeclareMathOperator{\D}{\mathcal{D}}
\newcommand{\dd}{d_A}
\newcommand{\dM}{d}

\newcommand{\g}{\mathfrak{g}}
\newcommand {\emptycomment}[1]{}
\newcommand{\standardsign}{\mathrm{sgn}}

\newcommand{\wedgeA}{\Gamma(\wedge^\bullet A)}

\newcommand{\CinfM}{C^\infty(M)}
\newcommand{\secA}{\Gamma(A)}
\newcommand{\secB}{\Gamma(B)}
\newcommand{\secAd}{\Gamma(A^*)}
\newcommand{\OmegaA}{\Gamma(\wedge^\bullet A^*)}
\newcommand{\OmegaAdegree}[1]{\Gamma(\wedge^{#1} A^*)}
\newcommand{\wedgeAdegree}[1]{\Gamma(\wedge^{#1} A)}
\newcommand{\oldpsi}{\phi}

\newcommand{\Brg}[2]{[ #1 , #2 ]_{\g}}
\newcommand{\Brgd}[2]{[ #1 , #2 ]_{\g^*}}
\newcommand{\BrA}[2]{[ #1 , #2 ]_A}
\newcommand{\BrAd}[2]{[ #1 , #2 ]_*}

\newcommand{\CBr}[2]{[ #1 , #2 ]}

\newcommand{\puredegree}[1]{|#1|}

\newcommand{\bigD}{\bar{D}}

\newcommand{\XX}{\mathfrak{X}}
\newcommand{\GammaL}{\Gamma\mathcal{L}}
\newcommand{\fotimes}{\tilde{\otimes}}
\newcommand{\trace}{\mathrm{tr}}
\newcommand{\equalbyreason}[1]{\xlongequal[]{\mbox{#1}}}

\newcommand{\tobefilledin}{\,\cdot\,}

\usepackage{stmaryrd}

\newcommand{\R}{\mathbb{R}}

\newcommand{\threesecofA}{\tau}
\newcommand{\dstar}{d_\ast}

\begin{document}
\title{Dirac generating operators of split Courant algebroids}

\author{Liqiang Cai}
\address{School of Mathematics and Statistics, Henan University}
\email{\href{mailto:cailiqiang@vip.henu.edu.cn}{cailiqiang@vip.henu.edu.cn}}

\author{Zhuo Chen}
\address{Department of Mathematics, Tsinghua University}
\email{\href{mailto:zchen@math.tsinghua.edu.cn}{zchen@math.tsinghua.edu.cn}}

\author{Honglei Lang}
\address{College of Science, China Agricultural University} 
\email{\href{mailto:hllang@cau.edu.cn}{hllang@cau.edu.cn}}

\author{Maosong Xiang}
\address{Center for Mathematical Sciences, Huazhong University of Science and Technology}
\email{\href{mailto: msxiang@hust.edu.cn}{msxiang@hust.edu.cn}}
\thanks{Research partially supported by NSFC grants 11701146, 11901221, 11901568, 12071241.}

\begin{abstract}
Given a vector bundle $A$ over a smooth manifold $M$ such that the square root $\mathcal{L}$ of the line bundle $\wedge^{\mathrm{top}}A^\ast \otimes \wedge^{\mathrm{top}}T^\ast M$ exists,
the Clifford bundle associated to the split pseudo-Euclidean vector bundle $(E = A \oplus A^\ast, \langle \cdot, \cdot \rangle)$, admits a spinor bundle $\wedge^\bullet A \otimes \mathcal{L}$, whose section space can be thought of as that of Berezinian half-densities of the graded manifold $A^\ast[1]$.
We give an explicit construction of Dirac generating operators of split Courant algebroid (or proto-bialgebroid) structures on $A \oplus A^\ast$ introduced by Alekseev and Xu. We also prove that the square of the Dirac generating operator gives rise to an invariant of the split Courant algebroid.
\end{abstract}

\maketitle

\tableofcontents

\section*{Introduction}
The main object of this paper is split Courant algebroid, which can also be expressed as a proto-bialgebroid.
In \cite{CS}, the second author of this paper and Sti\'{e}non studied a special type of split Courant algebroids which are doubles of Lie bialgebroids. This paper presents a study of this problem in general.

The origin of studying proto-bialgebroids can be traced back to Drinfeld's work on Lie bialgebras \cite{Drinfeld1983}. After that, with his landmark ``Quantum groups'' article \cite{Drinfeld1986}, a series of follow-up studies developed rapidly.
For example,   Drinfeld  further studied quasi-Hopf algebras  that generalize the Hopf
algebras defining quantum groups, and their semi-classical limits, the Lie quasi-bialgebras \cite{Drinfeld1990}. Kosmann-Schwarzbach also introduced the notion of quasi-Poisson Lie groups  in \cites{YKS1991,YKS1992}. It turns out that the infinitesimal counterpart of a quasi-Poisson Lie group is a quasi-Lie bialgebra, which is another weak version of the Lie bialgebra structure. The more general case of proto-bialgebras (called there ``proto-Lie bialgebras'') is treated in  \cite{B-YKS1993}.

Lie bialgebras are special cases of Lie bialgebroids introduced by Mackenzie and Xu in~\cite{MX1994}, where they appeared as linearization of Poisson groupoids. Likewise, proto-bialgebras are special cases of proto-bialgebroids. According to  Kosmann-Schwarzbach~\cite{KS2005} (see also \cites{Roy1999, Roy2002}), a proto-bialgebroid is a pair $(A,A^\ast)$ of vector bundles in duality, together with a degree $3$  function $\Theta$  (known as the Hamiltonian generating function) on the $(-2)$-shifted Poisson manifold $T^\ast[2]A[1]$ satisfying the classical master equation $\{\Theta,\Theta\}=0$.
Unpacking this Hamiltonian generating function $\Theta$, we obtain the following data:
\begin{itemize}
\item Two skew-symmetric brackets $\BrA{\tobefilledin}{\tobefilledin}$, $\BrAd{\tobefilledin}{\tobefilledin}$  on $\Gamma(A)$  and  $\Gamma(A^\ast)$, respectively;
\item Two bundle maps (called anchors) $a_A\colon A \to TM$ and $a_\ast \colon A^\ast \to TM$;
\item  Two elements ${\threesecofA} \in \wedgeAdegree{3}$ and $\oldpsi \in\OmegaAdegree{3}$, which are ``fluxes'' in field theory \cite{DeserStasheff2014}.
\end{itemize}
Unpacking the classical master equation $\{\Theta,\Theta\}=0$, both $(A, \BrA{\tobefilledin}{\tobefilledin}, a_A)$ and $(A^\ast, \BrAd{\tobefilledin}{\tobefilledin}, a_\ast)$ are \textit{skew-symmetric dull algebroids} in the sense of \cite{MJ2018}, and the above data are subject to several compatibility conditions (see Definition~\ref{Def:Proto-bialgebroids}).
When both ${\threesecofA}$ and $\oldpsi$  vanish, it becomes a Lie bialgebroid $(A, A^\ast)$.
The cases  $\oldpsi=0$ or $\threesecofA=0$ correspond to Lie quasi-bialgebroids or quasi-Lie bialgebroids, respectively.

The notion of Courant algebroids also originates from Drinfeld's observation \cite{Drinfeld1986} that the direct sum $\g\oplus \g^\ast$ (called the Drinfeld double) of a Lie bialgebra $(\g,\g^\ast)$ is a canonical quadratic Lie algebra.
Extending the construction of the Drinfeld double of a Lie bialgebra to the
case of a Lie bialgebroid $(A,A^\ast)$ is a non-trivial problem. One of the solutions is
provided by Liu, Weinstein and Xu \cite{LWX} in terms of the \textit{Courant algebroid} structure on $A\oplus A^\ast$.
Roughly speaking, a Courant algebroid consists of a pseudo-Euclidean vector bundle $(E, \langle\cdot,\cdot\rangle)$ over a smooth manifold $M$, a Leibniz bracket $\circ\colon \Gamma(E)\times \Gamma(E)\to \Gamma(E)$ (known as the Dorfman bracket), and a map $\rho\colon E\to TM$ (called the anchor) satisfying several compatible conditions (see Definition \ref{Def:CA-Dorfman bracket}).

Proto-bialgebroids can also be interpreted as \textit{split Courant algebroids}.
Given a proto-bialgebroid $(A,A^\ast)$, one can obtain a Courant algebroid structure on $A\oplus A^\ast$ (see Section~\ref{Sec:Pre} for detail).
Conversely, to get a proto-bialgebroid out of a Courant algebroid $E$, one needs an additional assumption --- the pseudo-Euclidean vector bundle $E$ decomposes as the direct sum of two
transverse Lagrangian sub-bundles. In other words, when the Courant structure is defined on the Whitney sum $E=A\oplus A^\ast$ of a vector bundle $A$ and its dual, where $A$  and $A^\ast$ are both co-isotropic subbundles in $E$, we obtain a split Courant algebroid, and thus a proto-bialgebroid structure on $(A,A^\ast)$.
In particular, the Courant algebroid structure on $E$ can be the double of a Lie bialgebroid, a Lie quasi-bialgebroid, or a quasi-Lie bialgebroid.

Various attempts have been made to understand Courant algebroids.
One method is due to Weinstein, \u{S}evera, and Roytenberg \cites{Severaletter, Roy2001}  --- a Courant algebroid can be described as a degree $2$ symplectic graded manifold together with a degree $3$  Hamiltonian generating function $\Theta$ satisfying $\{\Theta,\Theta\} = 0$, where $\{\cdot,\cdot\}$ is the graded Poisson bracket induced from the graded symplectic structure. This graded Poisson bracket is called big bracket in~\cite{KS2005}. The anchor map and the Dorfman bracket of the Courant algebroid $E$ are recovered as derived brackets.

Around the same time,  in an unpublished manuscript \cite{AX2001manu},  motivated by an earlier work of Cabras and Vinogradov \cite{CabrasVinogradov}, Alekseev and Xu approached Courant algebroids in terms of \textit{Dirac generating operators}, an analogue of Kostant's cubic Dirac operators \cite{Kostant1999}. Here is a quick sketch --- Let $E\to M$ be a vector bundle endowed with a fiberwise nondegenerate pseudo-metric $\langle \cdot, \cdot\rangle$, and let $\mathcal{C}(E)$ be the associated bundle of \textit{Clifford algebras}. Assume that there exists a bundle of Clifford modules $S$ over the same base manifold $M$, that is, a vector bundle whose fibers are Clifford modules over fibers of $\mathcal{C}(E)$. The natural $\mathbb{Z}_2$-grading of $\Gamma(\mathcal{C}(E))$ induces a $\mathbb{Z}_2$-grading on the operators on $S$.  For example,  the multiplication by a function $f\in \CinfM$ is an even operator, while the Clifford action of a section $e\in \Gamma(E)$ is odd.  A Dirac generating operator  is an odd operator $D$ on $\Gamma(S)$ satisfying the following properties (here and below, $[\cdot,\cdot]$ stands for the graded commutator on the space of graded operators on $\Gamma(S)$):
\begin{itemize}
\item For all $f\in \CinfM$, the operator $[D, f]$ is the Clifford action of some section of $E$.
\item For all $e_1,e_2\in \Gamma(E)$, the operator $[[D, e_1], e_2]$ is the Clifford action of some section of $E$.
\item The square of $D$ is the multiplication by some function on M.
\end{itemize}
From a Dirac generating operator $D$,  the \textit{derived bracket} $e_1\circ e_2 = [[D, e_1], e_2]$ on $\Gamma(E)$ together with the anchor map $\rho \colon E\to TM$ given by $\rho(e)f=2\langle[D, f], e\rangle$,  define a Courant algebroid structure on $E$.
Conversely,
it is proved \emph{loc.cit} that for a general Courant algebroid $E$ there exists a  Dirac generating operator, acting on a certain spinor bundle of $(E, \langle\cdot,\cdot\rangle)$, which plays exactly the same role as the de Rham differential operator does in the Cabras-Vinogradov's approach to the standard Courant algebroid.

Part of the motivation behind this work is to better understand the Dirac generating
operators of Courant algebroids associated to proto-bialgebroids.
As stated earlier that,  the proto-bialgebroid structure on $(A, A^\ast)$ can be encoded in a certain Hamiltonian function $\Theta$ satisfying the classical master equation $\{\Theta, \Theta\} = 0$. According to~\cites{KS2005, Roy1999, Roy2002}, the function $\Theta$ is the sum of four homogeneous terms
 \[
\Theta=\dd+d_\ast+\threesecofA+\oldpsi\in C^\infty(T^\ast[2]A[1])
\]
 where $\dd$ and $d_\ast$ correspond to the skew-symmetric dull algebroid structure on $A$ and $A^\ast$, respectively. The main purpose of this paper is to prove a \textit{quantum
analog} of this condition.

Here is an outline of our results. Given a rank $n$ vector bundle $A$ over an $m$-dimensional smooth manifold $M$, we consider the vector bundle $E=A\oplus A^\ast$ which is equipped with the standard pseudo-metric (see \eqref{Eqt:standardmetricsplitcase}). The spin module we take is $S=\wedge^\bullet A \otimes (\wedge^nA^\ast \otimes\wedge^mT^\ast M)^{1/2}$,  which can be thought of as the space of Berezinian half densities on the graded manifold $A^\ast[1]$.
Given a pair of skew-symmetric dull algebroids $A$ and $A^\ast$, and two elements ${\threesecofA} \in\wedgeAdegree{3}$ and $\oldpsi \in\OmegaAdegree{3}$, we introduce an operator
\begin{eqnarray*}
\bigD:=\breve{{\dstar}}+\breve{\partial}+{\threesecofA}-\iota_{\oldpsi} \colon \Gamma(S) \rightarrow  \Gamma(S).
\end{eqnarray*}
Here $\breve{{\dstar}}$ and $\breve{\partial}$  come from the skew-symmetric dull algebroid structure on $A^\ast$ and $A$, respectively, similar to how we define a Batalin-Vilkovisky operator.
For detailed explanation of the symbols, see Section \ref{Sec:firstmainthmsection}. The operator $\bigD$ actually comes from a formula invented by Kosmann-Schwarzbach  in \cite{KS2005},  called a \textit{deriving operator} therein.
Our first main theorem (Theorem \ref{MAIN THM}) declares the following equivalence of facts:
\[
(A, A^\ast)~\mbox{forms a proto-bialgebroid} ~\Leftrightarrow~  \bigD^2\in \CinfM ~\Leftrightarrow~ \bigD \mbox{ is a Dirac generating operator}.
\]
Recently, Gr\"{u}tzmann, Michel, and Xu~\cite{GMX} studied Weyl quantization of degree 2 symplectic graded manifolds. Given a pseudo-Euclidean vector bundle $(E, \langle \cdot,\cdot \rangle)$ over $M$, each metric connection $\nabla$ on $E$ determines a degree 2 symplectic graded manifold $(T^\ast[2]M \oplus E[1], \omega_\nabla)$. They proved that the Weyl quantization of this degree $2$ symplectic graded manifold establishes a bijection between Hamiltonian generating functions and \textit{skew-symmetric} Dirac generating operators.  By considering the square of the unique skew-symmetric Dirac generating operator, they also obtain a new Courant algebroid invariant. This new invariant, as a function on the base manifold, is a natural extension of the square norm of the Cartan 3-form of a quadratic Lie algebra.

In our second main result (Theorem \ref{f-D2}), we give a specific expression of the Dirac generating operator $\bigD$ of a proto-bialgebroid $(A, A^\ast)$,
and prove directly that the square $\bigD^2$ of $\bigD$
is indeed the invariant of the proto-bialgebroid $(A,A^\ast)$ (without Weyl quantization).
When reduced to the case $\threesecofA=0$ and $\oldpsi=0$, our results recover the conclusions in~\cite{CS} regarding Dirac generating operators for Lie bialgebroids.

The derived brackets of the Courant algebroids  and more generally, metric algebroids \cite{Vaisman2012},  play an important role in the generalized complex geometry developed by Hitchin \cite{Hitchin1999} and Gualtieri \cite{Gualtieri2007}, and double field theory \cites{DeserStasheff2014,DFT-MSS}, where many remarkable results have been established. We hope our results will be of some use in this subject.
Notably, it is shown in the papers \cites{A-Costa2020,FegierZambon2015} that each split Courant algebroid $E=A\oplus A^\ast$ corresponds to a multiplicative curved $L_\infty$-algebra structure on $\Gamma(\wedge^\bullet A)[2]$.

The paper is organized as follows. Section \ref{Sec:Pre} gives a succinct account of standard facts about Courant algebroids, proto-bialgebroids and Dirac generating operators. The differential operator $\bigD$ is defined in Section \ref{Sec:Dirac4split} and the main theorems are then stated without proofs. Section \ref{Sec:Proofs} is devoted to prove
the statements in Section \ref{Sec:Dirac4split}. Our results are then particularized to a few concrete situations in Section \ref{Sec:examples}.

\subsection*{Conventions, terminologies and notations}
\smallskip
\begin{enumerate}
	\item \textit{The manifold $M$, the ring $ \CinfM$, and space of vector fields $\XX(M)$.} We only work with real smooth manifolds, say $M$.
The symbol $\CinfM$ denotes the algebra of real valued smooth functions on $M$;
and $\XX(M):=\Gamma(TM)$ denotes the space of vector fields.
	\item \textit{The tensor product} $\fotimes$ stands for $\otimes_{\CinfM}$.
	\item \textit{Vector bundle $A \rightarrow M$.} A vector bundle $A \rightarrow M$ means a real vector bundle of finite rank. Denote by $A^\ast$ the dual vector bundle of $A$.
	\item  \textit{Skew-symmetric dull algebroids $A$ and $A^\ast$.} The terminology of skew-symmetric  dull algebroids is introduced in \cites{MJ2018}, which refers to a triple  $(A,\BrA{\tobefilledin}{\tobefilledin},a_A)$  consisting of the following data
	\begin{itemize}
		\item a vector bundle $A\rightarrow M$;
		\item a bundle map $a_A\colon \secA\rightarrow\XX(M)$, called the anchor;
		\item an $\R$-bilinear and skew-symmetric bracket $\BrA{\tobefilledin}{\tobefilledin} \colon \secA\times\secA\rightarrow\secA$
	\end{itemize}
	satisfying  the Leibniz rule
	\begin{eqnarray*}
		~\BrA{x}{fy}=f\BrA{x}{y}+a_A(x)(f)\, y,
	\end{eqnarray*}
	for all $x,y\in\secA$ and $f\in \CinfM$.
	Moreover, if the Jacobi identity for $\BrA{\tobefilledin}{\tobefilledin}$ holds (that is $\BrA{\tobefilledin}{\tobefilledin}$ is a Lie bracket),
	then the skew-symmetric dull algebroid $(A,\BrA{\tobefilledin}{\tobefilledin},a_A)$ is a Lie algebroid.
	
	Similarly,   $(A^\ast,\BrAd{\tobefilledin}{\tobefilledin},a_\ast)$ denotes a skew-symmetric dull algebroid whose underlying vector bundle is dual to $A$.  We do not presume any compatibility conditions between $(A,\BrA{\tobefilledin}{\tobefilledin},a_A)$ and $(A^\ast, \BrAd{\tobefilledin}{\tobefilledin}, a_\ast)$ unless otherwise specified.
	\item \textit{The derivations $\dd$ and ${\dstar}$.} Given a skew-symmetric dull algebroid $(A,\BrA{\tobefilledin}{\tobefilledin},a_A)$, it induces a derivation
	\[
	\dd\colon  \OmegaAdegree{\bullet} \to \OmegaAdegree{\bullet+1}
   \]
   by
	\begin{align*}
		({\dd}\omega)(x_0,x_1,\ldots,x_n) &:= \sum_{i=0}^n (-1)^i a_A(x_i)
\big(\omega(x_0,x_1,\ldots,\hat{x_i},\ldots,x_n)\big)\\
		&\qquad+\sum_{0\leqslant i< j\leqslant n} (-1)^{i+j} \omega(\BrA{x_i}{x_j},x_0,x_1,\ldots,\hat{x_i},\ldots,\hat{x_j},\ldots,x_n),
	\end{align*}
for all $\omega\in \OmegaAdegree{n}$ and $x_0,x_1,\ldots, x_n\in\secA$.
It follows that $\dd^2=0$ if and only if $(A,\BrA{\tobefilledin}{\tobefilledin},a_A)$ is a Lie algebroid.
The derivation of the skew-symmetric dull algebroid $(A^\ast,\BrAd{\tobefilledin}{\tobefilledin},a_\ast)$ is denoted by
\[
{\dstar}\colon \wedgeAdegree{\bullet} \rightarrow \wedgeAdegree{\bullet+1}.
\]
\item  \textit{Pairings $ \langle \tobefilledin|\tobefilledin\rangle $ and $\langle \tobefilledin,\tobefilledin\rangle $ of $A\oplus A^\ast$.}
For any $\xi\in \secAd$, denote by $\iota_{\xi}\colon \wedgeAdegree{\bullet} \rightarrow \wedgeAdegree{\bullet-1}$ the standard contraction defined by
\[
(\iota_{\xi}r)(\eta_1,\eta_2,\ldots,\eta_{n-1}) :=r(\xi,\eta_1,\eta_2,\ldots,\eta_{n-1}),
	\quad \forall r\in \wedgeAdegree{n}, \eta_1,\ldots,\eta_{n-1}\in\secAd.
\]
For $\xi_1\wedge\xi_2\wedge\ldots\wedge\xi_n\in \OmegaAdegree{n}$, we define
\[
\iota_{\xi_1\wedge\xi_2\wedge\ldots\wedge\xi_n} :=\iota_{\xi_n}\circ \iota_{\xi_{n-1}}\circ \cdots\circ \iota_{\xi_1} \colon \wedgeAdegree{\bullet} \rightarrow \wedgeAdegree{\bullet-n}.
\]
Similarly, for all $x,x_1,\cdots, x_n \in \secA$, we have contractions $\iota_x \colon \OmegaAdegree{\bullet} \rightarrow \OmegaAdegree{\bullet-1}$ and
\[
\iota_{x_1\wedge x_2\wedge\ldots\wedge x_n}
	:=\iota_{x_n}\circ\iota_{x_{n-1}}\circ\cdots\circ\iota_{x_1}  \colon \Gamma(\wedge^\bullet A^\ast) \to  \Gamma(\wedge^{\bullet-n} A^\ast).
\]
We make the following agreement:
	\begin{eqnarray*}
		\langle \xi_1 \wedge \ldots \wedge \xi_n| x_1 \wedge \ldots \wedge x_n \rangle
		&:=& \iota_{x_n} \iota_{x_{n-1}}\cdots \iota_{x_1} (\xi_1 \wedge \ldots \wedge \xi_n)\\
		&:=& \iota_{\xi_n} \iota_{\xi_{n-1}}\cdots \iota_{\xi_1} (x_1 \wedge \ldots \wedge x_n)\\
		&=&\sum_{\sigma \in S_n} \standardsign(\sigma)
		  \xi_1 ( x_{\sigma(1)} )\xi_2 ( x_{\sigma(2)} ) \cdots   \xi_n ( x_{\sigma(n)} ).
	\end{eqnarray*}
There are two pairings on $A\oplus A^\ast$ given by for all $x+\xi,y+\eta\in\Gamma(A\oplus A^\ast)$,
	\begin{eqnarray*}
		\langle x+\xi|y+\eta\rangle &:=&\xi(y)+\eta(x),	\\
		\mbox{ and }\quad
		\langle x+\xi,y+\eta\rangle &:=& \frac{1}{2}\langle x+\xi|y+\eta\rangle = \frac{1}{2}\xi(y)+\frac{1}{2}\eta(x).
	\end{eqnarray*}
	 We refer to the second one as the \textit{standard pseudo-metric} on $A\oplus A^\ast$.
\item \textit{The elements ${\threesecofA}$ and $\oldpsi$.}
Throughout this paper, the symbol ${\threesecofA}$ stands for an element in $\wedgeAdegree{3}$; while   $\oldpsi$ stands for an element in $ \OmegaAdegree{3}$.
\item \textit{Einstein convention} is adopted throughout the paper: $a^ib_i=\sum_i a^ib_i$.	
\end{enumerate}
\paragraph{\bf Acknowledgements}
{We would like to thank Zhangju Liu,  Mathieu Sti\'enon, and  Ping Xu   for fruitful discussions and useful comments.
 }

\section{Preliminaries}\label{Sec:Pre}
In this preliminary section, we make a succinct introduction to split Courant algebroids, proto-bialgebroids, and Dirac generating operators.
\subsection{Split Courant algebroids and proto-bialgebroids}
The notion of Courant algebroid was first introduced in \cite{LWX}.
\begin{Def}\label{Def:CA-Dorfman bracket}
A Courant algebroid is a vector bundle $E \rightarrow M$ equipped with three structures: (1)
a pseudo-metric on $E$, i.e., a nondegenerate symmetric bilinear form $\langle \tobefilledin,\tobefilledin\rangle$ on $\Gamma(E)$; (2) a bilinear operation $\circ$ on $\Gamma(E)$ called Dorfman bracket; and (3) a bundle map $\rho\colon E \rightarrow TM$ called anchor. These structure maps are subject to the following axioms:
	\begin{compactenum}
		\item  $e_1 \circ (e_2 \circ e_3)=(e_1 \circ e_2) \circ e_3+e_2 \circ (e_1 \circ e_3), \;\forall e_1,e_2,e_3\in\Gamma(E)$;
		\item $\rho(e_1 \circ e_2)=[\rho(e_1),\rho(e_2)], ~\forall e_1,e_2\in\Gamma(E);$
		\item $e_1 \circ (f e_2) =f(e_1 \circ e_2)+\rho(e_1)(f) \, e_2, ~\forall e_1,e_2\in\Gamma(E), f\in \CinfM;$
		\item $e \circ e=\D\langle e,e\rangle, ~\forall e\in\Gamma(E);$
		\item $\rho(e) \, \langle h_1,h_2\rangle=\langle e \circ h_1,h_2\rangle+\langle h_1,e \circ h_2\rangle, ~\forall e,h_1,h_2\in\Gamma(E)$,
	\end{compactenum}
	where $\D\colon \CinfM\rightarrow \Gamma(E)$ is defined by\footnote{Note that the definition of   $\D$ in different literature may differ by a constant multiple.}
 \begin{equation*}\label{Eqt:Df}
		\langle \D(f),e\rangle=\frac{1}{2}\rho(e)(f).
\end{equation*}
\end{Def}
This paper is devoted to study a particular type of Courant algebroids commonly known as  split Courant algebroids. In fact, they are equivalent to the objects of proto-bialgebroids. Let us firstly clarify this correspondence.

A pseudo-Euclidean vector bundle  is called \textbf{split}, if it is isomorphic to the Whitney sum $A \oplus A^\ast$ for some vector bundle $A \to M$, equipped with the standard pseudo-metric given by
\begin{equation}\label{Eqt:standardmetricsplitcase}
\langle x+\xi, y+\eta \rangle := \frac{1}{2}\langle x+\xi|y+\eta\rangle
=\frac{1}{2}\xi(y)+\frac{1}{2}\eta(x), \quad\forall x+\xi,y+\eta\in \Gamma(A\oplus A^\ast).
\end{equation}
A Courant algebroid is called split, if its underlying pseudo-Euclidean vector bundle is split.

Assume that $E = A \oplus A^\ast$ is a split Courant algebroid. Then the anchor $\rho \colon E \to TM$ is decomposed into $a_A\colon A\to TM$ and $a_{\ast}\colon A^\ast\to TM$ by
\begin{equation*}
	\rho(x+\xi)=a_A(x)+a_{\ast}(\xi),
\end{equation*}
for all $x \in\secA$ and $\xi \in\secAd$.
The restriction of the Dorfman bracket on $\Gamma(A)$ determines a skew-symmetric dull algebroid $(A, \BrA{\tobefilledin}{\tobefilledin}, a_A)$ and an element $\oldpsi\in\OmegaAdegree{3}$ by
\[
x\circ y=\BrA{x}{y}-\iota_y\iota_{x}\oldpsi,
\]
for all $x,y\in\secA$.
Similarly, the restriction of the Dorfman bracket on $\Gamma(A^\ast)$ determines  skew-symmetric dull algebroid $(A^\ast,\BrAd{\tobefilledin}{\tobefilledin}, a_\ast)$ and an element  ${\threesecofA} \in \wedgeAdegree{3}$ by
\[
\xi\circ \eta=-\iota_\eta\iota_{\xi}{\threesecofA}+\BrAd{\xi}{\eta},
\]
for all $\xi,\eta\in\secAd$. The Dorfman bracket $\circ$ on $\Gamma(A \oplus A^\ast)$ takes the form
\begin{eqnarray}\label{Eqn:Dorfmanbracket}
	(x+\xi)\circ(y+\eta)
	&=(\BrA{x}{y}+L_\xi y -\iota_\eta ({\dstar}(x)) -\iota_\eta\iota_{\xi}{\threesecofA})
+(\BrAd{\xi}{\eta}+L_x\eta-\iota_y(\dd(\xi))-\iota_y\iota_{x}\oldpsi ).
\end{eqnarray}
The pair $(A, A^\ast)$ of skew-symmetric dull algebroids is subject to several compatibility conditions induced from the axioms of the Courant algebroid $A\oplus A^\ast$. (See~\cite{Roy1999} for details and also~\cite{KS2005} in terms of big bracket.)
These conditions are exactly the axioms of proto-bialgebroids, which we summarize below.
\begin{Def}[\cites{Roy1999, KS2005}]\label{Def:Proto-bialgebroids}
A proto-bialgebroid consists of the following data:
\begin{itemize}
\item two skew-symmetric dull algebroids in duality, $(A,\BrA{\tobefilledin}{\tobefilledin},a_A)$ and $(A^\ast, \BrAd{\tobefilledin}{\tobefilledin}, a_\ast)$,
		\item an element ${\threesecofA} \in\wedgeAdegree{3}$,
		\item an element $\oldpsi \in\OmegaAdegree{3}$.
	\end{itemize}
	They are subject to the following axioms: for all $x,y,z\in\secA$, $\xi,\eta,\chi\in\secAd$,
	\begin{compactenum}
		\item The Jacobi identity of $\BrA{\tobefilledin}{\tobefilledin}$ is controlled by $\oldpsi$ and ${\dstar}$, i.e.,
		\begin{eqnarray*}\label{Jacobi of A1}
			\BrA{\BrA{x}{y}}{z} +\BrA{\BrA{y}{z}}{x} +\BrA{\BrA{z}{x}}{y} ={\dstar}\big(\oldpsi(x,y,z)\big)
			+\iota_{\oldpsi}\big({\dstar} (x\wedge y\wedge z)\big),
		\end{eqnarray*}
where the map  ${\dstar}\colon \wedgeAdegree{\bullet} \rightarrow \wedgeAdegree{\bullet+1}$ is the derivation arising from the skew-symmetric dull algebroid $(A^\ast,\BrAd{\tobefilledin}{\tobefilledin}, a_\ast)$.
\item The Jacobi identity of $\BrAd{\tobefilledin}{\tobefilledin}$ is controlled by ${\threesecofA}$ and $\dd$, i.e.,
		\begin{eqnarray*}
			\BrAd{\BrAd{\xi}{\eta}}{\chi}+\BrAd{\BrAd{\eta}{\chi}}{\xi}+\BrAd{\BrAd{\chi}{\xi}}{\eta}
			=\dd\big({\threesecofA}(\xi,\eta,\chi)\big) +\iota_{{\threesecofA}} \big(\dd(\xi\wedge\eta\wedge \chi)\big),
		\end{eqnarray*}		
where the map $\dd\colon \OmegaAdegree{\bullet} \rightarrow \OmegaAdegree{\bullet+1}$ is the derivation arising from the skew-symmetric dull algebroid  $(A, \BrA{\tobefilledin}{\tobefilledin}, a_A)$.		
		\item The skew-symmetric dull algebroids
		$(A,\BrA{\tobefilledin}{\tobefilledin},a_A)$ and $(A^*,\BrAd{\tobefilledin}{\tobefilledin},a_*)$ are compatible in the sense that
		\begin{eqnarray*}
	{\dstar}(\BrA{x}{y}) =\BrA{{\dstar}(x)}{y} +\BrA{x}{{\dstar}(y)} +\iota_{(\iota_y\iota_x\oldpsi)}{\threesecofA} ;
		\end{eqnarray*}
		\item The element $\oldpsi\in \OmegaAdegree{3}$ is $\dd$-closed, i.e., $\dd(\oldpsi) =0$.
		\item The element ${\threesecofA}\in\wedgeAdegree{3}$ is ${\dstar}$-closed, i.e., ${\dstar}({\threesecofA}) =0$.
\end{compactenum}
Such a proto-bialgebroid will be denoted by $(A,\BrA{\tobefilledin}{\tobefilledin}, \BrAd{\tobefilledin}{\tobefilledin}, a_A, a_{\ast},{\threesecofA},\oldpsi)$.
\end{Def}
\begin{Rem}
Assume that $(A,\BrA{\tobefilledin}{\tobefilledin}, \BrAd{\tobefilledin}{\tobefilledin}, a_A, a_{\ast}, {\threesecofA},\oldpsi )$ is a proto-bialgebroid.
\begin{itemize}
\item If ${\threesecofA} \in \wedgeAdegree{3}$ vanishes, then the triple $(A^\ast, \BrAd{\tobefilledin}{\tobefilledin}, a_{\ast})$ is a Lie algebroid and the six-tuple $(A, \BrA{\tobefilledin}{\tobefilledin}, \BrAd{\tobefilledin}{\tobefilledin},a_A,a_{\ast}, \oldpsi )$ is known as a quasi-Lie bialgebroid.
\item If $\oldpsi\in \OmegaAdegree{3}$ vanishes, then the triple $(A,\BrA{\tobefilledin}{\tobefilledin}, a_A)$ is a Lie algebroid and the six-tuple $(A,\BrA{\tobefilledin}{\tobefilledin}, \BrAd{\tobefilledin}{\tobefilledin}, a_A, a_{\ast}, {\threesecofA} )$ is called a Lie quasi-bialgebroid.
\item If both $\threesecofA$ and $\oldpsi$ vanish, then $A$ and $A^\ast$ form a Lie bialgebroid.
\end{itemize}
\end{Rem}

\subsection{Dirac generating operators}
We now briefly recall the approach to Courant algebroids via Dirac generating operators~\cites{AX2001manu,GMX}.
Given a pseudo-Euclidean vector bundle $(E, \langle \cdot, \cdot \rangle)$ over $M$,
let $\mathcal{C}(E)\rightarrow M$ be the associated bundle of Clifford algebras with the generating relation $e_1\otimes e_2+e_2\otimes e_1 = 2\langle e_1,e_2\rangle$, for all $p \in M$ and all $e_1,e_2\in\mathcal{C}(E)_p$.
Assume that there exists a smooth vector bundle $S\rightarrow M$ whose fiber $S_p$ over every point $p\in M$ is the spin module of the Clifford algebra $\mathcal{C}(E)_p$. Assume further that $S$ is $\mathbb{Z}_2$-graded, i.e., $S=S^0\oplus S^1$.
An operator $D$ on $\Gamma(S)$ is said to be even (resp. odd) if $D(S^i)\subset S^i$ (resp. $D(S^i)\subset S^{i+1}$). Here $i \in \mathbb{Z}_2$.
If $D_1$ and $D_2$ are operators of degree $i_1$ and $i_2$, respectively, then their graded commutator is given by $\CBr{D_1}{D_2}=D_1\circ D_2-(-1)^{i_1i_2}D_2\circ D_1$.
\begin{Def}[\cite{AX2001manu}]\label{Def:DGO}
A Dirac generating operator for a pseudo-Euclidean vector bundle $(E,\langle\tobefilledin,\tobefilledin\rangle)$ is an odd operator $D$ on $\Gamma(S)$ satisfying the following conditions.
\begin{compactenum}[$(a)$]
\item For all $f\in \CinfM$, we have $\CBr{D}{f}\in\Gamma(E)$. This means that the operator $\CBr{D}{f}$ is the Clifford action of some section of $E$ on $\Gamma(S)$.
\item For all $e_1,e_2\in\Gamma(E)$, we have $\CBr{\CBr{D}{e_1}}{e_2}\in\Gamma(E)$.
\item The square of $D$ is the multiplication by some smooth function on $M$, i.e., $D^2\in \CinfM$.
\end{compactenum}
\end{Def}
\begin{Thm}[\cite{AX2001manu}]
Let $D$ be a Dirac generating operator for a pseudo-Euclidean vector bundle $(E,\langle\tobefilledin,\tobefilledin\rangle)$ over $M$. Then there is a canonical Courant algebroid structure on $E$,  whose anchor and Dorfman bracket are defined respectively by
		\begin{eqnarray*}
			\rho(e)(f)&=&\CBr{\CBr{D}{f}}{e}, \\
			\quad\mbox{ and }\quad
			e_1\circ e_2&=&\CBr{\CBr{D}{e_1}}{e_2},
		\end{eqnarray*}
where $e, e_1, e_2 \in \Gamma(E)$ and $f\in \CinfM$.
\end{Thm}

It is natural to ask which kind of Dirac generating operators generates split Courant algebroids. This question is our main concern of the paper.
	
\section{Dirac generating operators of split Courant algebroids}\label{Sec:Dirac4split}
 In this section, we characterize Dirac generating operators of split Courant algebroids.
 We mainly follow the approach developed in~\cite{CS}, where Dirac generating operators of Courant algebroids arising from Lie bialgebroids are considered.

\subsection{General settings and the first main theorem}\label{Sec:firstmainthmsection}
Let $(A,\BrA{\tobefilledin}{\tobefilledin}, a_A)$ be a skew-symmetric dull algebroid, and $B$ a vector bundle over the same base manifold $M$.
By saying an $A$-connection on $B$, we mean an $\R$-bilinear map
\[
\secA \times \secB\to \secB,\quad (x,b)\mapsto \nabla_x  b
\]
satisfying
\begin{align*}
\nabla_{fx} b &= f \nabla_x b  &\mbox{and} & & \nabla_x  (fb) &= a_A(x)(f)b+f \nabla_x b,
\end{align*}
for all $x\in\secA, b\in \secB, f\in \CinfM$.
Such an $A$-connection determines an operator called covariant derivative $\dd^B\colon \OmegaAdegree{\bullet}\fotimes \secB \rightarrow \OmegaAdegree{\bullet+1} \fotimes \secB$ satisfying
\[
\dd^B(\omega \otimes b) = \big(\dd(\omega)\big)\otimes b +(-1)^{k}\omega\wedge \big(\dd^B(b)\big),
\]
for all $\omega\in \OmegaAdegree{k}, b\in \secB$. Here and in the sequel $\fotimes$ stands for $\otimes_{\CinfM}$. However, an element in $\OmegaAdegree{\bullet}\fotimes \secB$ is still denoted by $\omega\otimes b$ rather than $\omega\fotimes b$.

Let $(A,\BrA{\tobefilledin}{\tobefilledin}, a_A)$ and $(A^\ast,\BrAd{\tobefilledin}{\tobefilledin},a_\ast)$ be a pair of skew-symmetric dull algebroids, where $A$ is a rank $n$ vector bundle over an $m$-dimensional manifold $M$.
Inspired by~\cite{ELW},  the line bundle $\wedge^nA^\ast \otimes\wedge^mT^\ast M$ admits a canonical $A^\ast$-connection. More precisely, a section $\xi\in\secAd$ ``acts" on $\Gamma(\wedge^nA^\ast \otimes\wedge^mT^\ast M)$ by Lie derivatives
\begin{eqnarray*}
L_\xi  (\xi_1\wedge\ldots\wedge\xi_n\otimes\mu)
=\sum_{i=1}^n(\xi_1\wedge\ldots\wedge\BrAd{\xi}{\xi_i}\wedge\ldots\wedge\xi_n\otimes\mu)
+\xi_1\wedge\ldots\wedge\xi_n\otimes L_{a_\ast(\xi)}\mu.
\end{eqnarray*}
The square root $\mathcal{L}=(\wedge^nA^\ast \otimes\wedge^mT^\ast M)^{1/2}$, if exists,  also admits an $A^\ast$-connection. Denote the associated covariant derivative by
\begin{eqnarray*}
\breve{{\dstar}}\colon \wedgeAdegree{k}\fotimes\GammaL \rightarrow  \wedgeAdegree{k+1}\fotimes \GammaL .
\end{eqnarray*}
Similarly, if the square root $(\wedge^n A \otimes \wedge^m T^\ast M)^{1/2}$  exists, it also admits an $A$-connection, and thus a covariant derivative
\begin{eqnarray*}
\Gamma\big(\wedge^k A^\ast \otimes (\wedge^n A \otimes \wedge^m T^\ast M)^{1/2}\big)
\rightarrow
\Gamma\big(\wedge^{k+1} A^\ast \otimes (\wedge^n A \otimes \wedge^m T^\ast M)^{1/2}\big).
\end{eqnarray*}
Using the isomorphisms of vector bundles
\begin{eqnarray*}
\wedge^k A^\ast \cong \wedge^k A^\ast \otimes \wedge^{n} A \otimes \wedge^{n} A^\ast
\cong \wedge^{n-k} A \otimes \wedge^n A^\ast
\end{eqnarray*}
and
\begin{eqnarray*}
\wedge^n A^\ast \otimes (\wedge^n A \otimes \wedge^m T^\ast M)^{1/2}
\cong (\wedge^n A^\ast \otimes \wedge^m T^\ast M)^{1/2},
\end{eqnarray*}
one has a family of isomorphisms
\begin{eqnarray*}
\beta_k \colon \wedge^k A^\ast \otimes (\wedge^n A \otimes \wedge^m T^\ast M)^{1/2}
&\cong& \wedge^{n-k}A \otimes (\wedge^n A^\ast \otimes \wedge^m T^\ast M)^{1/2}.
\end{eqnarray*}
Therefore, one ends up with a derivation
\begin{eqnarray*}
\breve{\partial}\colon \wedgeAdegree{k}\fotimes\GammaL \rightarrow  \wedgeAdegree{k-1}\fotimes\GammaL .
\end{eqnarray*}
In the sequel, we fix two elements ${\threesecofA}\in\wedgeAdegree{3}$ and $\oldpsi\in\OmegaAdegree{3}$. They give rise to two operators
\begin{eqnarray*}
{\threesecofA}\colon \wedgeAdegree{k}\fotimes\GammaL
\rightarrow \wedgeAdegree{k+3}\fotimes\GammaL,
&\quad&
{\threesecofA}(r\otimes l):=({\threesecofA}\wedge r)\otimes l,
\\\mbox{and }\quad
\iota_\oldpsi\colon \wedgeAdegree{k}\fotimes\GammaL
\rightarrow\wedgeAdegree{k-3}\fotimes\GammaL,
&\quad&
\iota_\oldpsi(r\otimes l):=(\iota_\oldpsi r)\otimes l.
\end{eqnarray*}
We wish to find a simple condition so that ${\threesecofA}$ and $\oldpsi$ together with the two skew-symmetric dull algebroids $A$ and $A^\ast$ form a proto-bialgebroid.
For this purpose, the spinor bundle we choose is $S = \wedge^\bullet A\otimes \mathcal{L}$.
For the convenience of description, we introduce the odd operator
\begin{eqnarray}\label{hatD}
\bigD=\breve{{\dstar}}+\breve{\partial}+{\threesecofA}-\iota_{\oldpsi}\colon {\wedgeA\fotimes \GammaL } \rightarrow{\wedgeA\fotimes \GammaL } .
\end{eqnarray}
Here is our first main result, which can be regarded as an enhancement of \cite{CS}*{Theorem 3.1}.
\begin{Thm}\label{MAIN THM}
Let $(A,\BrA{\tobefilledin}{\tobefilledin},a_A)$ and $(A^\ast, \BrAd{\tobefilledin}{\tobefilledin}, a_\ast)$ be two dual skew-symmetric dull algebroids. Given ${\threesecofA}\in\wedgeAdegree{3}$ and  $\oldpsi\in\OmegaAdegree{3}$, let $\bigD$ be the operator defined in~\eqref{hatD}.
Then the following three statements are equivalent:
\begin{compactenum}
\item The septuple $(A,\BrA{\tobefilledin}{\tobefilledin}, \BrAd{\tobefilledin}{\tobefilledin}, a_A, a_\ast, {\threesecofA}, \oldpsi)$ is a proto-bialgebroid;
\item The square of the operator $\bigD$ satisfies $\bigD^2\in \CinfM$;
\item The operator $\bigD\colon {\wedgeA\fotimes \GammaL } \rightarrow{\wedgeA\fotimes \GammaL } $ is a Dirac generating operator for the split pseudo-Euclidean vector bundle $(A\oplus A^\ast, \langle\cdot,\cdot\rangle)$, where $ \langle\cdot,\cdot\rangle$ is the standard metric defined in~\eqref{Eqt:standardmetricsplitcase}.
\end{compactenum}
\end{Thm}
Suppose that $\bigD$ is a Dirac generating operator. The function $\bigD^2\in \CinfM$ is called the \textbf{characteristic function} of the proto-bialgebroid$(A, \BrA{\tobefilledin}{\tobefilledin}, \BrAd{\tobefilledin}{\tobefilledin}, a_A, a_\ast, {\threesecofA},\oldpsi)$ or of the associated split Courant algebroid $(A\oplus A^\ast, \langle\tobefilledin,\tobefilledin\rangle,\circ,\rho,{\threesecofA},\oldpsi)$.
\emptycomment{
As a by-product of  the Dirac generating operator $\bigD$, the Courant algebroid structure on $A\oplus A^\ast$ are of derived form:
\begin{eqnarray*}
\mbox{the anchor} \qquad
\rho(e)(f)&=&\CBr{\CBr{\bigD}{f}}{e},
\\
\mbox{the Dorfman bracket}\qquad
e_1\circ e_2&=&\CBr{\CBr{\bigD}{e_1}}{e_2},
\end{eqnarray*}
where $e, e_1, e_2\in\Gamma(A\oplus A^\ast), f\in \CinfM$.
}
The proof of Theorem  \ref{MAIN THM} is deferred to Section \ref{Sec:Proofs}.

\subsection{Modular elements and the second main theorem}\label{Sec:LAP}
 Throughout this section, we assume that $(A,\BrA{\tobefilledin}{\tobefilledin},a_A)$ and $(A^\ast, \BrAd{\tobefilledin}{\tobefilledin}, a_\ast)$ are skew-symmetric dull algebroids in duality. Denote by
$\dd\colon \OmegaAdegree{\bullet}\rightarrow\OmegaAdegree{\bullet+1}$ and
$\dstar \colon \wedgeAdegree{\bullet}\rightarrow\wedgeAdegree{\bullet+1}$ the associated derivations.
The Lie derivative along any element $x \in \secA$ is an $\R$-linear derivation
\[
L_x \colon \Gamma(\wedge^\bullet A \otimes \wedge^\diamond A^\ast) \to \Gamma(\wedge^\bullet A \otimes \wedge^\diamond A^\ast)
\]
induced by $L_xy := [x,y]_A$ for all $y \in \secA$ and by a Cartan type formula $L_x \xi := \dd (\iota_x\xi) + \iota_x (\dd \xi)$.
Similarly, the Lie derivative along any element $\xi \in \secAd$ is an $\R$-linear derivation
\[
L_\xi \colon \Gamma(\wedge^\bullet A \otimes \wedge^\diamond A^\ast) \to \Gamma(\wedge^\bullet A \otimes \wedge^\diamond A^\ast)
\]
induced by $L_\xi \eta := [\xi,\eta]_\ast$ for all $\eta \in \secAd$ and by a Cartan type formula $L_\xi x := d_\ast(\iota_\xi x) + \iota_\xi (d_\ast x)$.
\begin{lem}[\cite{CS}*{Proposition 4.6}]\label{42}
For all $x\in\secA,\xi\in\secAd$ and $f\in \CinfM$, we have
\begin{align*}
\big\langle {\dstar}\big(\BrA{f}{x}\big) -\BrA{{\dstar}(f)}{x} +\BrA{f}{{\dstar}(x)}|\xi\big\rangle
&= \langle (L_{d_A f}  + L_{d_\ast f})x | \xi \rangle  =\big(\CBr{L_x}{L_\xi}-L_{x\circ\xi}\big)(f),
\end{align*}
where $x \circ \xi := -\iota_\xi (d_\ast x) + L_x \xi$.
\end{lem}
Let us choose a nowhere-vanishing section $\Omega\in \OmegaAdegree{n}$ and let $V \in \wedgeAdegree{n}$ be the dual of $\Omega\in \OmegaAdegree{n}$ in the sense that $\langle\Omega|V\rangle=1$.
\begin{lem}[\cite{CS}*{Lemmas 4.1 and 4.3}]
For all $x\in\secA$ and $\xi\in\secAd$, we have
\begin{eqnarray}\label{LLL1}
L_x\Omega&=&-\partial(x)\,\Omega, \\
\label{LLL2}L_x V&=&\partial(x)\, V,\\
\label{LLL3}L_\xi\Omega&=&\partial_*(\xi)\,\Omega,\\
\label{LLL5}(L_x\Omega)\otimes V&=&-\Omega\otimes(L_x V),\\
\label{LLL6}(L_\xi\Omega)\otimes V&=&-\Omega\otimes (L_\xi V).
\end{eqnarray}
\end{lem}
The elements $\Omega$ and $V$ induce two isomorphisms:
\begin{eqnarray*}
\Omega^\sharp&\colon &\Gamma(\wedge^kA) \rightarrow \Gamma(\wedge^{n-k}A^\ast), \qquad   r\mapsto \iota_r\Omega, \\
V^\sharp&\colon &\Gamma(\wedge^kA^\ast) \rightarrow \Gamma(\wedge^{n-k}A),\qquad \omega \mapsto\iota_\omega V,
\end{eqnarray*}
which are essentially inverse to each other:
\begin{eqnarray*}
V^\sharp\circ\Omega^\sharp
&=&(-1)^{k(n-1)}\id_{\wedgeAdegree{k}}, \\
\Omega^\sharp\circ V^\sharp
&=&(-1)^{k(n-1)}\id_{\OmegaAdegree{k}} .
\end{eqnarray*}
Consider the operator $\partial$ induced from $\dd$ by the isomorphism $V^\sharp$:
\[
\begin{tikzcd}
\OmegaAdegree{k} \ar{r}{V^\sharp} \ar{d}[swap]{-(-1)^k\dd} &\wedgeAdegree{n-k}
\ar{d}{\partial} \\
\OmegaAdegree{k+1} \ar{r}[swap]{V^\sharp} & \wedgeAdegree{n-k-1}.
\end{tikzcd}
\]
In other words, for all $\omega\in\OmegaAdegree{k}$, we have
\begin{eqnarray}\label{Eqn:V-d-partial}
-V^\sharp\dd\omega =(-1)^k\partial V^\sharp \omega,
\end{eqnarray}
which implies that
\begin{eqnarray*}
\partial r =  -(-1)^{n-k}\left(V^\sharp \circ \dd \circ (V^\sharp)^{-1}\right) r = - (-1)^{n(k+1)}(V^\sharp \circ \dd \circ \Omega^\sharp) r, \quad\forall r\in\wedgeAdegree{k}.
\end{eqnarray*}
The operator $\partial$ is not a derivation, but a second order differential operator,  called a Batalin-Vilkovisky operator for the skew-symmetric dull algebroid $A$.
For any $r_1\in\wedgeAdegree{k}$ and $r_2\in\wedgeAdegree{l}$, we have the BV relation
\begin{eqnarray}\label{BV}
\BrA{r_1}{r_2}=(-1)^k\partial(r_1\wedge r_2) -(-1)^k (\partial r_1)\wedge r_2 -r_1\wedge(\partial r_2).
\end{eqnarray}

Similarly, we also have the Batalin-Vilkovisky operator $\partial_\ast$ dual to ${\dstar}$:
\[
\begin{tikzcd}
\wedgeAdegree{n-k} \ar{d}[swap]{(-1)^k{\dstar}} &\OmegaAdegree{k} \ar{l}[swap]{V^\sharp}
\ar{d}{\partial_\ast} \\
\wedgeAdegree{n-k+1} & \OmegaAdegree{k-1} \ar{l}{V^\sharp}.
\end{tikzcd}
\]
Explicitly, we have for all $\omega\in\OmegaAdegree{k}$,
\begin{eqnarray*}
{\dstar} V^\sharp\omega =(-1)^kV^\sharp\partial_\ast\omega.
\end{eqnarray*}
The BV relation for $\partial_\ast$  reads
\begin{eqnarray}\label{BV2}
\BrAd{\Xi_1}{\Xi_2} =(-1)^l\partial_\ast(\Xi_1\wedge\Xi_2) -(-1)^l(\partial_\ast\Xi_1)\wedge\Xi_2 - \Xi_1 \wedge(\partial_\ast\Xi_2),
\end{eqnarray}
 for all $\Xi_1\in\OmegaAdegree{l}$, $\Xi_2\in \OmegaAdegree{p}$.
\begin{Rem}
One should be cautious that in general $\partial^2\neq 0$. In fact, from Equation \eqref{BV} one can find for all $x,y,z\in\secA$,
\begin{eqnarray*}
\partial^2(x\wedge y) &=&\BrA{\partial(x)}{y} +\BrA{x}{\partial(y)} -\partial(\BrA{x}{y}),\\
\mbox{and}\quad
\partial^2(x\wedge y\wedge z) &=&\frac{1}{3}\partial\big(\BrA{x\wedge y}{z}\big)
+\frac{1}{3}\partial\big(\BrA{y\wedge z}{x}\big) +\frac{1}{3}\partial\big(\BrA{z\wedge x}{y}\big) \\
&&\qquad+\frac{1}{3}\partial\big(x\wedge \BrA{y}{z}\big) +\frac{1}{3}\partial\big(y\wedge \BrA{z}{x}\big)
+\frac{1}{3}\partial\big(z\wedge \BrA{x}{y}\big)\\
&&\qquad\qquad+\BrA{x}{\partial(y) z} +\BrA{y}{\partial(z) x} +\BrA{z}{\partial(x) y} \\
&&\qquad\qquad\qquad+\BrA{\partial(x)}{y} \,z +\BrA{\partial(y)}{z} \,x +\BrA{\partial(z)}{x} \,y.
\end{eqnarray*}
In fact, if $\partial^2=0$, then $A$ is a Lie algebroid.
\end{Rem}

 \begin{def-prop}
Let $s\in\Gamma(\wedge^m T^\ast M)$ be a volume form of $M$, $\Omega\in \OmegaAdegree{n}$  a nowhere-vanishing section, and $V \in \wedgeAdegree{n}$ be its dual such that $\langle\Omega | V \rangle=1$.
 \begin{compactenum}
 		\item There exists a unique $X_0 \in \secA$, called the modular element of the skew-symmetric dull algebroid $A^\ast$, such that
 		\begin{eqnarray}\label{X_0}
 	L_\xi(\Omega\otimes s) =(L_\xi\Omega)\otimes s +\Omega\otimes L_{a_\ast(\xi)}s =\langle \xi | X_0\rangle\Omega\otimes s, \quad\forall\xi\in\secAd.
 		\end{eqnarray}
 		\item There exists a unique $\xi_0\in\secAd$, called the modular element  of the skew-symmetric dull algebroid $A$, such that
 		\begin{eqnarray}\label{xi_0}
 		L_x(s\otimes V) =(L_{a_A(x)}s)\otimes V+s\otimes L_xV =\langle \xi_0 | x \rangle s\otimes V, \quad\forall x\in\secA.
 		\end{eqnarray}
 \end{compactenum}
 \end{def-prop}
When $(A,A^\ast)$ is a Lie bialgebroid, both $X_0$ and $\xi_0$ are Chevalley-Eilenberg $1$-cocycles, called modular cocycles. Their cohomology classes are called modular classes~\cite{ELW}.

Our second main result  gives expression of the characteristic function $\bigD^2$  when the operator $\bigD$ in~\eqref{hatD} is a Dirac generating operator.
\begin{Thm}\label{f-D2}
Assume that $(A,\BrA{\tobefilledin}{\tobefilledin},\BrAd{\tobefilledin}{\tobefilledin}, a_A, a_\ast, {\threesecofA}, \oldpsi)$ is a proto-bialgebroid such that the line bundle $\wedge^nA^\ast \otimes\wedge^mT^\ast M$ of the graded manifold $A^\ast[1]$ admits a square root $\mathcal{L}$.
Then the Dirac generating operator $\bigD$ has the form
\begin{eqnarray}\label{expression of D}
\bigD ={\dstar}+\frac{1}{2}X_0 - \partial + \frac{1}{2}\iota_{\xi_0} + {\threesecofA}-\iota_{\oldpsi} \colon  {\wedgeA\fotimes \GammaL } \to {\wedgeA\fotimes \GammaL } .
\end{eqnarray}
Moreover, the characteristic function
 \begin{equation}\label{Eqt:brevefexplicitly}
 \breve{f}:= \bigD^2 = \frac{1}{4}\langle\xi_0|X_0\rangle -\frac{1}{2}\partial(X_0) -\langle{\threesecofA}|\oldpsi\rangle,
\end{equation}
is an invariant of this proto-bialgebroid, which can be determined by any of the following two equations:
\begin{compactenum}
\item[$(a)$] $L_{X_0}(\Omega\otimes s)
=4\big(\breve{f}+\langle{\threesecofA}|\oldpsi\rangle\big) \, \Omega\otimes s$;
\item[$(b)$] $L_{\xi_0}(s\otimes V)
=4\big(\breve{f}+\langle{\threesecofA}|\oldpsi\rangle\big) \, s\otimes V$.
\end{compactenum}
\end{Thm}
We postpone the proof to Section \ref{Sec:Proofs}.
In particular, when both $\tau$ and $\phi$ vanishes, we rediscover the Dirac generating operator and characteristic function of Lie bialgebroids.
\begin{Cor}[\cites{CS, GMX}]
  Assume that $(A,\BrA{\tobefilledin}{\tobefilledin},\BrAd{\tobefilledin}{\tobefilledin}, a_A, a_\ast)$ is a Lie bialgebroid. Then its Dirac generating operator $\bigD$ has the form
  \[
   \bigD =  \breve{{\dstar}}+\breve{\partial} ={\dstar}-\partial+\frac{1}{2}(X_0 +\iota_{\xi_0}).
  \]
  The characteristic function
  \[
  \breve{f} = \frac{1}{4}\langle\xi_0|X_0\rangle -\frac{1}{2}\partial(X_0)
  \]
  is an invariant of this Lie bialgebroid $(A,A^\ast)$.
\end{Cor}

\section{Proofs of main theorems}\label{Sec:Proofs}
Let  $(A,\BrA{\tobefilledin}{\tobefilledin},a_A)$ and $(A^\ast, \BrAd{\tobefilledin}{\tobefilledin}, a_\ast)$ be two skew-symmetric dull algebroids in duality.
We also fix two elements ${\threesecofA}\in\wedgeAdegree{3}$ and $\oldpsi\in\OmegaAdegree{3}$.
All other notations are specified as before.

\subsection{Some preparatory work}
\begin{lem}\label{Lem:temp327}
For all $s \in \wedgeAdegree{k}$ and $\omega \in \OmegaAdegree{n-k}$ satisfying $s=V^\sharp(\omega) = \iota_\omega V$, we have
\begin{eqnarray*}
&&\big(\partial^2
		-{\dstar} \circ\iota_{\oldpsi}
		-\iota_{\oldpsi}\circ{\dstar}
		-\frac{1}{2}\iota_{\iota_{X_0}\oldpsi}
		-\frac{1}{2}\iota_{\dd(\xi_0)}\big)(s)
		\nonumber\\
		&=&-V^\sharp\bigg(\big(\dd^2
		+L_{\oldpsi}-\partial_\ast(\oldpsi)
		+\frac{1}{2}\iota_{X_0}\oldpsi
		+\frac{1}{2}\dd(\xi_0)\big)\omega\bigg),
\end{eqnarray*}	
where $L_{\oldpsi}\omega=\BrAd{\oldpsi}{\omega}$.
\end{lem}
\begin{proof}
Firstly, using~\eqref{Eqn:V-d-partial} twice, we have
\begin{eqnarray*}
	\partial^2(s) &= \partial^2\big(V^\sharp(\omega)\big) =-V^\sharp(\dd^2\omega).
\end{eqnarray*}
Direct computations show that
\begin{align*}
\iota_{\iota_{X_0}\oldpsi}s &= \iota_{\iota_{X_0}\oldpsi}V^\sharp(\omega)
		= \iota_{(\iota_{X_0}\oldpsi)\wedge\omega}V
		=V^\sharp\big((\iota_{X_0}\oldpsi)\wedge\omega\big),
\end{align*}
and
\begin{align*}
	\iota_{\dd(\xi_0)}s &= \iota_{\dd(\xi_0)}V^\sharp(\omega) =  \iota_{\big(\dd(\xi_0)\big)\wedge\omega}V =  V^\sharp\bigg(\big(\dd(\xi_0)\big)\wedge\omega\bigg).
\end{align*}
Meanwhile,
\begin{align*}
({\dstar} \circ\iota_{\oldpsi}+\iota_{\oldpsi}\circ{\dstar} )(s) &=({\dstar} \circ\iota_{\oldpsi} +\iota_{\oldpsi}\circ{\dstar} ) \big(V^\sharp(\omega)\big)\\
			&= {\dstar} (\iota_{\omega\wedge\oldpsi}V) +\iota_{\oldpsi}\big((-1)^{n-k} V^\sharp\partial_\ast\omega\big)\\
			&= {\dstar}  V^\sharp(\omega\wedge\oldpsi) +(-1)^{n-k}\iota_{\oldpsi} \iota_{\partial_\ast\omega}V\\
			&=-(-1)^{n-k}V^\sharp\partial_\ast(\omega\wedge\oldpsi) +(-1)^{n-k}V^\sharp \big((\partial_\ast\omega)\wedge\oldpsi\big)\\			&=-V^\sharp\big(\partial_\ast(\oldpsi\wedge\omega)+\oldpsi\wedge(\partial_\ast\omega)\big)\\
			&=V^\sharp\bigg(\big(L_{\oldpsi}-\partial_\ast(\oldpsi)\big)\omega\bigg),
	\end{align*}
where the final equality follows from the BV relation~\eqref{BV2}.
Combing these four equations, we conclude the proof.
\end{proof}

We introduce two operators $K$ and $L$.
\begin{align*}
K(x,y) \colon& \secA\rightarrow\secA, &
K(x,y)(z) &:=\iota_\oldpsi\big({\dstar} (x) \wedge y \wedge z\big)
-\iota_\oldpsi\big(x \wedge {\dstar} (y) \wedge z\big),
\\\mbox{and ~}~
L(\xi,\eta) \colon& \secAd\rightarrow\secAd, &
L(\xi,\eta)(\chi) &:=\iota_{{\threesecofA}(\eta,\chi)}\big(\dd(\xi)\big)
+\iota_{{\threesecofA}(\chi,\xi)}\big(\dd(\eta)\big),
\end{align*}
where $x,y,z\in\secA,\xi,\eta,\chi\in\secAd$.
\begin{lem}\label{KL-K}
For $K$ and $L$ as defined  above, we have
\begin{eqnarray*}
\trace\big(K(x,y)\big)
&=&-2\big\langle{\dstar}(x)|\iota_y\oldpsi\big\rangle
+2\big\langle{\dstar}(y)|\iota_x\oldpsi\big\rangle,\\
\trace\big(L(\xi,\eta)\big) &=&-2\big\langle\dd(\xi)|\iota_\eta{\threesecofA}\big\rangle
+2\big\langle\dd(\eta)|\iota_\xi{\threesecofA}\big\rangle,
\end{eqnarray*}
for all $x,y\in\secA,\xi,\eta\in\secAd$.
\end{lem}
\begin{proof}
Let $\{\xi^1,\ldots,\xi^n\}$ be a local basis of $\secA$ and $\{\theta_1,\ldots,\theta_n\}$ be the dual basis of $\secAd$. Then we have
\begin{eqnarray*}
\trace\big(K(x,y)\big)
&=&\big\langle K(x,y)(\xi^a)|\theta_a\big\rangle\\
&=&-\big\langle \iota_{\theta_a}\big({\dstar}(x)\big)|\iota_{\xi^a}\iota_y\oldpsi\big\rangle
+\big\langle \iota_{\theta_a}\big({\dstar}(y)\big)|\iota_{\xi^a}\iota_x\oldpsi\big\rangle\\
&=&-2\big\langle{\dstar}(x)|\iota_y\oldpsi\big\rangle
+2\big\langle{\dstar}(y)|\iota_x\oldpsi\big\rangle.
\end{eqnarray*}
The second equation for $\trace\big(L(\xi,\eta)\big)$ can be similarly examined.
\end{proof}

\begin{lem}
Assume that  $(A,\BrA{\tobefilledin}{\tobefilledin}, \BrAd{\tobefilledin}{\tobefilledin}, a_A, a_\ast, {\threesecofA},\oldpsi)$ is a proto-bialgebroid. Then
\begin{align}
\big(\dd(\xi_0)(x,y)\big)\Omega\otimes s
&= (L_{\oldpsi(x,y)}\Omega)\otimes s
-\Omega\otimes L_{a_*(\oldpsi(x,y))}(s)
+2\left(\big\langle{\dstar}(x)|\iota_y\oldpsi\big\rangle
-\big\langle{\dstar}(y)|\iota_x\oldpsi\big\rangle\right)\Omega\otimes s
,\nonumber\\
\label{X0-sV}
\big({\dstar}(X_0)(\xi,\eta)\big)s\otimes V
&= s\otimes (L_{{\threesecofA}(\xi,\eta)}V)
-(L_{a_A({\threesecofA}(\xi,\eta))}s)\otimes V
+2\left(\big\langle\dd(\xi)|\iota_\eta{\threesecofA}\big\rangle
-\big\langle\dd(\eta)|\iota_\xi{\threesecofA}\big\rangle\right) s\otimes V,
\end{align}
for all $x,y\in\secA,\xi,\eta\in\secAd$.
\end{lem}
\begin{proof}
We first prove
\begin{eqnarray}\label{xi0-1}
\big( \dd(\xi_0) (x,y)\big)s\otimes V &=&-(L_{\BrA{x}{y}}-L_xL_y+L_yL_x)(s\otimes V).
\end{eqnarray}
In fact, by Equations \eqref{X_0} and \eqref{xi_0}, we have
\begin{eqnarray*}
&&(L_{\BrA{x}{y}}-L_xL_y+L_yL_x)(s\otimes V)\\
&=&\langle\xi_0|\BrA{x}{y}\rangle s\otimes V - L_x\big(\langle\xi_0|y\rangle s\otimes V\big)
+L_y\big(\langle\xi_0|x\rangle s\otimes V\big) \\
&=&\langle\xi_0|\BrA{x}{y}\rangle s\otimes V - a_A(x)\big(\langle\xi_0|y\rangle\big) s\otimes V
-\langle\xi_0|y\rangle\langle\xi_0|x\rangle s\otimes V\\
&&\qquad+a_A(y)\big(\langle\xi_0|x\rangle\big) s\otimes V +\langle\xi_0|x\rangle\langle\xi_0|y\rangle s\otimes V \\
&=&\bigg(\langle\xi_0|\BrA{x}{y}\rangle - a_A(x)\big(\langle\xi_0|y\rangle\big)
+a_A(y)\big(\langle\xi_0|x\rangle\big)\bigg) s\otimes V \\
&=&-\big(\dd(\xi_0)(x,y)\big)s\otimes V.
\end{eqnarray*}
Since for all $x,y \in \secA$
\[
L_{\BrA{x}{y}} - L_xL_y+L_yL_x  = L_{\oldpsi(x,y)} + K(x,y) \colon \Gamma(\wedge^\bullet A) \to \Gamma(\wedge^\bullet A),
\]
it follows that
\begin{eqnarray*}
&&\big(\dd(\xi_0)\big)(x,y)(\Omega \otimes s)\otimes V \\
&\equalbyreason{\eqref{xi0-1}} &
-\Omega \otimes (L_{\BrA{x}{y}}-L_xL_y+L_yL_x)(s\otimes V)\\
&=&-\Omega \otimes \big((L_{a_A(\BrA{x}{y})}-L_{a_A(x)}L_{a_A(y)} + L_{a_A(y)} L_{a_A(x)})s\big)\otimes V\\
&&\quad-\Omega \otimes s\otimes\big((L_{\BrA{x}{y}}-L_xL_y+L_yL_x)V\big)\\
& = & -\Omega \otimes (L_{a_\ast(\oldpsi(x,y))}s)\otimes V -\Omega \otimes s\otimes\big((L_{\oldpsi(x,y)}+K(x,y))V\big)\\
&\equalbyreason{\eqref{LLL6}} & -\Omega \otimes (L_{a_\ast(\oldpsi(x,y))}s)\otimes V
+(L_{\oldpsi(x,y)}\Omega)\otimes s\otimes V -\trace\big(K(x,y)\big)\Omega \otimes s\otimes V\\
&\equalbyreason{Lemma~\ref{KL-K}} &
-\Omega \otimes (L_{a_\ast(\oldpsi(x,y))}s)\otimes V
+(L_{\oldpsi(x,y)}\Omega)\otimes s\otimes V\\
&&\qquad-\left(-2\big\langle{\dstar}(x)|\iota_y\oldpsi\big\rangle
+2\big\langle{\dstar}(y)|\iota_x\oldpsi\big\rangle\right)\Omega \otimes s\otimes V,
\end{eqnarray*}
Thus, we have
\[
\big(\dd(\xi_0)(x,y)\big)\Omega \otimes s =(L_{\oldpsi(x,y)}\Omega) \otimes s
-\Omega\otimes (L_{a_*(\oldpsi(x,y))}s) +2\left(\big\langle{\dstar}(x)|\iota_y\oldpsi\big\rangle
-\big\langle{\dstar}(y)|\iota_x\oldpsi\big\rangle\right)\Omega \otimes s.
\]
The second equation~\eqref{X0-sV} can be verified similarly and thus omitted.
\end{proof}

\begin{lem}\label{Main-prop-1}
Assume that $(A,\BrA{\tobefilledin}{\tobefilledin}, \BrAd{\tobefilledin}{\tobefilledin}, a_A, a_\ast, {\threesecofA},\oldpsi)$ is a proto-bialgebroid. Then
	\begin{eqnarray}\label{Main-prop-1-1}
	\partial({\threesecofA})  &=&\frac{1}{2}\iota_{\xi_0}{\threesecofA} +\frac{1}{2}{\dstar}(X_0),\\
		\label{Main-prop-1-2}
	\partial_\ast(\oldpsi) &=&\frac{1}{2}\iota_{X_0}\oldpsi +\frac{1}{2}\dd(\xi_0).
	\end{eqnarray}
\end{lem}
\begin{proof}
First of all, by Equation~\eqref{xi_0}, we have
\begin{eqnarray*}
				(\iota_{\xi_0}{\threesecofA})(\xi,\eta)s\otimes V
				&=&(L_{a_A({\threesecofA}(\xi,\eta))}s)\otimes V
				+s\otimes (L_{{\threesecofA}(\xi,\eta)}V).
\end{eqnarray*}
Using Equation \eqref{BV}, we obtain
\begin{eqnarray*}
\partial({\threesecofA}) &=&{\threesecofA}_1\wedge\BrA{{\threesecofA}_2}{{\threesecofA}_3}
+\BrA{{\threesecofA}_1}{{\threesecofA}_3}\wedge{\threesecofA}_2
-\BrA{{\threesecofA}_1}{{\threesecofA}_2}\wedge{\threesecofA}_3\nonumber\\
&&\qquad+\partial({\threesecofA}_1)\,{\threesecofA}_2\wedge{\threesecofA}_3
-\partial({\threesecofA}_2)\,{\threesecofA}_1\wedge{\threesecofA}_3
+\partial({\threesecofA}_3)\,{\threesecofA}_1\wedge{\threesecofA}_2.
\end{eqnarray*}
Hence,
\begin{eqnarray*}
		\partial({\threesecofA})(\xi,\eta)s\otimes V
				&\equalbyreason{\eqref{LLL2}} &\big({\threesecofA}_1\wedge\BrA{{\threesecofA}_2}{{\threesecofA}_3}
				+\BrA{{\threesecofA}_1}{{\threesecofA}_3}\wedge{\threesecofA}_2
				-\BrA{{\threesecofA}_1}{{\threesecofA}_2}\wedge{\threesecofA}_3\big)(\xi,\eta)s\otimes V\\
				&&\qquad+({\threesecofA}_2\wedge{\threesecofA}_3)(\xi,\eta) s\otimes L_{{\threesecofA}_1}V
				-({\threesecofA}_1\wedge{\threesecofA}_3)(\xi,\eta)s\otimes L_{{\threesecofA}_2}V\\
				&&\qquad\qquad+({\threesecofA}_1\wedge{\threesecofA}_2)(\xi,\eta)s\otimes L_{{\threesecofA}_3}V\\
				&=&\big({\threesecofA}_1\wedge\BrA{{\threesecofA}_2}{{\threesecofA}_3}
				+\BrA{{\threesecofA}_1}{{\threesecofA}_3}\wedge{\threesecofA}_2
				-\BrA{{\threesecofA}_1}{{\threesecofA}_2}\wedge{\threesecofA}_3\big)(\xi,\eta)s\otimes V\\
				&&+s\otimes L_{{\threesecofA}(\xi,\eta)} V
				+\bigg(a_A({\threesecofA}_1)\big(({\threesecofA}_2\wedge{\threesecofA}_3)(\xi,\eta)\big)
				-a_A({\threesecofA}_2)\big(({\threesecofA}_1\wedge{\threesecofA}_3)(\xi,\eta)\big)\\ &&\qquad\qquad+a_A({\threesecofA}_3)\big(({\threesecofA}_1\wedge{\threesecofA}_2)(\xi,\eta)\big)\bigg)
				s\otimes V\\
				&=& s\otimes L_{{\threesecofA}(\xi,\eta)} V
				+\big(\big\langle \dd(\xi)|\iota_\eta{\threesecofA}\big\rangle
				-\big\langle \dd(\eta)|\iota_\xi{\threesecofA}\big\rangle\big)s\otimes V.
			\end{eqnarray*}
		Here we have used the identity
\[
	L_{fx}V = fL_xV-a_A(x)(f)\, V,\quad\forall x\in \secA.
\]
Then using Equation \eqref{X0-sV}, we obtain
$\partial({\threesecofA}) =\frac{1}{2}\iota_{\xi_0}{\threesecofA} +\frac{1}{2}{\dstar}(X_0)$. 			
The proof of the second identity is similar and thus omitted.
\end{proof}

\begin{Cor}\label{Main-prop-22}
Assume that $(A,\BrA{\tobefilledin}{\tobefilledin}, \BrAd{\tobefilledin}{\tobefilledin}, a_A, a_\ast, {\threesecofA},\oldpsi)$ is a proto-bialgebroid. Then	
\begin{eqnarray}\label{Eqt:Main-prop-2}
\partial^2 = {\dstar}\circ\iota_{\oldpsi} +\iota_{\oldpsi}\circ{\dstar} +\frac{1}{2}\iota_{\iota_{X_0}\oldpsi}
	+\frac{1}{2}\iota_{\dd(\xi_0)}  \colon \wedgeAdegree{\bullet} \to \wedgeAdegree{\bullet-2}.
\end{eqnarray}
\end{Cor}
\begin{proof}
By assumption, the Jacobiator of the bracket $[\cdot,\cdot]_A$ is controlled by $\phi$ and $d_\ast$, i.e.,
 \[
 \dd^2+L_{\oldpsi}=0,
 \]
 where $L_{\oldpsi}=\BrAd{\oldpsi}{\cdot}$. Combining with Equation~\eqref{Main-prop-1-2} and the identity in Lemma \ref{Lem:temp327}, we obtain the desired result.
\end{proof}

Without loss of generality, suppose that ${\threesecofA}= \sum_i {\threesecofA}_{i_1} \wedge{\threesecofA}_{i_2} \wedge{\threesecofA}_{i_3}$ and $\oldpsi=\sum_j \oldpsi_{j_1}\wedge\oldpsi_{j_2}\wedge\oldpsi_{j_3}$, where ${\threesecofA}_i\in \secA $ and $\oldpsi_j \in \secAd  $.
We introduce four maps
\begin{eqnarray*}
Q_1 &:=&\sum_j\big( (\iota_{\oldpsi_{j_3}}\iota_{\oldpsi_{j_2}}{\threesecofA})\wedge \iota_{\oldpsi_{j_1}}
-(\iota_{\oldpsi_{j_3}}\iota_{\oldpsi_{j_1}}{\threesecofA})\wedge \iota_{\oldpsi_{j_2}}
+(\iota_{\oldpsi_{j_2}}\iota_{\oldpsi_{j_1}}{\threesecofA})\wedge \iota_{\oldpsi_{j_3}}\big)
\colon \wedgeAdegree{\bullet}\rightarrow\wedgeAdegree{\bullet},\\
Q_2 &:=&\sum_j\big( -(\iota_{\oldpsi_{j_3}}{\threesecofA})\wedge \iota_{\oldpsi_{j_2}}\iota_{\oldpsi_{j_1}} +(\iota_{\oldpsi_{j_2}}{\threesecofA})\wedge \iota_{\oldpsi_{j_3}}\iota_{\oldpsi_{j_1}} - (\iota_{\oldpsi_{j_1}}{\threesecofA})\wedge \iota_{\oldpsi_{j_3}}\iota_{\oldpsi_{j_2}}\big)
\colon \wedgeAdegree{\bullet}\rightarrow\wedgeAdegree{\bullet}, \\
Q_3 &:=&\sum_i\big( (\iota_{{\threesecofA}_{i_3}}\iota_{{\threesecofA}_{i_2}}\oldpsi)\wedge \iota_{{\threesecofA}_{i_1}} -(\iota_{{\threesecofA}_{i_3}}\iota_{{\threesecofA}_{i_1}}\oldpsi)\wedge \iota_{{\threesecofA}_{i_2}} +(\iota_{{\threesecofA}_{i_2}}\iota_{{\threesecofA}_{i_1}}\oldpsi)\wedge \iota_{{\threesecofA}_{i_3}}\big) \colon \OmegaAdegree{\bullet}\rightarrow\OmegaAdegree{\bullet},\\
Q_4 &:=&\sum_i\big( -(\iota_{{\threesecofA}_{i_3}}\oldpsi)\wedge \iota_{{\threesecofA}_{i_2}}\iota_{{\threesecofA}_{i_1}} + (\iota_{{\threesecofA}_{i_2}}\oldpsi)\wedge \iota_{{\threesecofA}_{i_3}}\iota_{{\threesecofA}_{i_1}} -(\iota_{{\threesecofA}_{i_1}}\oldpsi)\wedge \iota_{{\threesecofA}_{i_3}}\iota_{{\threesecofA}_{i_2}}\big)
\colon \OmegaAdegree{\bullet}\rightarrow\OmegaAdegree{\bullet}.
\end{eqnarray*}
\begin{lem}\label{Q1Q3}
The operators $Q_1$ and $Q_3$ satisfy
\begin{eqnarray*}
\big\langle Q_1(x)|\xi\big\rangle +\big\langle x|Q_3(\xi)\big\rangle
=2\langle \iota_x\oldpsi | \iota_\xi{\threesecofA}\rangle,
\end{eqnarray*}
for all $x\in\secA$ and $\xi\in\secAd$.
\end{lem}
\begin{proof}
For all $x\in\secA$ and $\xi\in\secAd$, we can examine that
\begin{eqnarray*}
\big\langle Q_1(x)|\xi\big\rangle
+\big\langle x|Q_3(\xi)\big\rangle
&=&\big\langle\big({\threesecofA}(\oldpsi_2,\oldpsi_3)\wedge \iota_{\oldpsi_1}
-{\threesecofA}(\oldpsi_1,\oldpsi_3)\wedge \iota_{\oldpsi_2}
+{\threesecofA}(\oldpsi_1,\oldpsi_2)\wedge \iota_{\oldpsi_3}\big)x|\xi\big\rangle\\
&&\qquad+\big\langle x|\big(\oldpsi({\threesecofA}_2,{\threesecofA}_3)\wedge \iota_{{\threesecofA}_1}
-\oldpsi({\threesecofA}_1,{\threesecofA}_3)\wedge \iota_{{\threesecofA}_2}
+\oldpsi({\threesecofA}_1,{\threesecofA}_2)\wedge \iota_{{\threesecofA}_3}\big)\xi\big\rangle
\\
&=&\langle x\wedge{\threesecofA}|\oldpsi\wedge\xi\rangle
+\langle x|\xi\rangle\langle{\threesecofA}|\oldpsi\rangle
+\langle \xi\wedge\oldpsi|{\threesecofA}\wedge x\rangle
+\langle\xi|x\rangle\langle\oldpsi|{\threesecofA}\rangle
\\
&=&2\langle \iota_x\oldpsi | \iota_\xi{\threesecofA}\rangle.
\end{eqnarray*}
\end{proof}

The following lemma is a generalization of \cite{CS}*{Lemma 5.1}.
\begin{lem}\label{Lemma 5.1}
Assume that  $(A,\BrA{\tobefilledin}{\tobefilledin}, \BrAd{\tobefilledin}{\tobefilledin}, a_A,a_\ast, {\threesecofA},\oldpsi)$ is a proto-bialgebroid. Then
\begin{eqnarray*}
\big\langle y|\big(L_{x\circ\xi}-\CBr{L_x}{L_\xi}\big)(\eta)\big\rangle
=\big\langle \iota_\eta{\dstar}(x)|\iota_y\dd(\xi)\big\rangle
+\oldpsi\big(x,{\threesecofA}(\xi,\eta),y\big),
\end{eqnarray*}
for all $x,y\in\secA$ and $\xi,\eta\in\secAd$.
\end{lem}
\begin{proof}
From the definition of the Dorfman bracket, it follows that
\begin{eqnarray*}
\langle y|L_{z+\chi}\eta\rangle
=2\big\langle y, (z+\chi)\circ\eta
+\iota_\eta{\dstar}(z)+{\threesecofA}(\chi,\eta)
\big\rangle
=2\big\langle y, (z+\chi)\circ\eta\big\rangle
\end{eqnarray*}
for all $z+\chi\in\Gamma(A\oplus A^\ast)$. Hence, we obtain
\begin{eqnarray*}
\langle y|L_xL_\xi\eta\rangle &=&2\big\langle y,x\circ(\xi\circ\eta)\big\rangle
-\oldpsi\big(x,{\threesecofA}(\xi,\eta),y\big)
\end{eqnarray*}
and
\begin{eqnarray*}
\langle y|L_\xi L_x\eta\rangle
&=&2\big\langle y,\xi\circ(L_x\eta)\big\rangle\\
&=&2\big(L_\xi\langle y,L_x\eta\rangle
-\langle \xi\circ y,L_x\eta\rangle\big)\\
&=&2\big(L_\xi\langle y,x\circ\eta\rangle
-\big\langle\xi\circ y,x\circ\eta+\iota_\eta{\dstar}(x)\big\rangle\big)\\
&=&2\big\langle y,\xi\circ(x\circ\eta)\big\rangle
+\big\langle \iota_\eta{\dstar}(x)| \iota_y\dd(\xi)\big\rangle,
\end{eqnarray*}
where we have used the equation
\begin{eqnarray*}
\big\langle y,L_\xi(x\circ\eta)\big\rangle =\big\langle y,L_\xi(L_x\eta-\iota_\eta{\dstar}(x))\big\rangle
 =\big\langle y,\xi\circ\big(L_x\eta-\iota_\eta{\dstar}(x)\big) +{\threesecofA}(\xi,L_x\eta)
-\iota_{\iota_\eta{\dstar}(x)}\dd(\xi)\big\rangle.
\end{eqnarray*}
Thus, we have
\begin{align*}
\big\langle y|(L_{x\circ\xi}-\CBr{L_x}{L_\xi})(\eta)\big\rangle
&= 2\big\langle y,(x\circ\xi)\circ\eta\big\rangle -2\big\langle y,x\circ(\xi\circ\eta)\big\rangle
+\oldpsi\big(x,{\threesecofA}(\xi,\eta),y\big)\\
&\qquad+2\big\langle y,\xi\circ(x\circ\eta)\big\rangle
+\big\langle \iota_\eta{\dstar}(x)| \iota_y\dd(\xi)\big\rangle\\
&=\big\langle \iota_\eta{\dstar}(x)| \iota_y\dd(\xi)\big\rangle +\oldpsi\big(x,{\threesecofA}(\xi,\eta),y\big).
\end{align*}
\end{proof}

Consider the pair of operators
\begin{align*}
D &={\dstar}+\partial \colon \wedgeA \to \wedgeA, & D_\ast &=\dd + \partial_\ast \colon \OmegaA \to \OmegaA.
\end{align*}
In the case of Lie bialgebroids, their squares yield a pair of Laplacian operators (see \cite{CS}):
\begin{align*}
\Delta &:= {\dstar}\partial+\partial{\dstar} \colon \wedgeAdegree{k} \rightarrow\wedgeAdegree{k},\\
\Delta_\ast &:= \dd\partial_\ast +\partial_\ast\dd \colon \OmegaAdegree{k} \rightarrow\OmegaAdegree{k}.
\end{align*}
In our situation, we also call them \textbf{Laplacians} although we do not require $(A,A^\ast)$ to be a Lie bialgebroid. They satisfy the following key relations.
\begin{lem}[\cite{CS}*{Proposition 4.5}]
For all $r_1,r_2\in\wedgeA$, we have
\begin{eqnarray}\label{prop 4.5}
&&\Delta(r_1\wedge r_2)-(\Delta r_1)\wedge r_2 -r_1\wedge(\Delta r_2) \nonumber\\
&=&(-1)^{\puredegree{r_1}}\big({\dstar}\big(\BrA{r_1}{r_2}\big) -\BrA{{\dstar}(r_1)}{r_2}
+(-1)^{\puredegree{r_1}}\BrA{r_1}{{\dstar}(r_2)} \big).
\end{eqnarray}
For all $x\in\secA,\xi\in\secAd, f,g\in \CinfM$, we have
\begin{eqnarray}
\label{41}\Delta_\ast(fg)-f\Delta_\ast(g)-g\Delta_\ast(f)&=&-\BrAd{\dd(f)}{g}+\BrAd{f}{\dd(g)},\\
\label{411}\Delta_\ast(f\xi)-f\Delta_\ast(\xi)-\Delta_\ast(f)\,\xi &=& \dd(\BrAd{f}{\xi})- \BrAd{\dd(f)}{\xi} + \BrAd{f}{\dd(\xi)}.
\end{eqnarray}
\end{lem}

\begin{prop}\label{Main-pre-thm-2} 	
Let $(A,\BrA{\tobefilledin}{\tobefilledin},a_A)$ and $(A^*,\BrAd{\tobefilledin}{\tobefilledin},a_*)$ be two skew-symmetric dull algebroids in duality. Then for ${\threesecofA}\in\wedgeAdegree{3}$ and $\oldpsi\in\OmegaAdegree{3}$, the following four assertions are equivalent:
\begin{itemize}
		\item[$\mathrm{(I)}$] For all $ x\in\secA$ and $\xi\in\secAd$,
		$\Delta_*\big(\langle\xi | x\rangle\big)
		=\big\langle\Delta_*(\xi) | x\big\rangle
		+\big\langle\xi | \Delta(x)\big\rangle
		-2\big\langle \iota_x\oldpsi | \iota_\xi{\threesecofA}\big\rangle$;
		
		\item[$\mathrm{(II)}$] For all $x\in\secA$ and $\xi\in\secAd$,
		$\Delta\big(\langle\xi | x\rangle\big)
		=\big\langle\Delta_*(\xi) | x\big\rangle
		+\big\langle\xi | \Delta(x)\big\rangle
		-2\big\langle \iota_x\oldpsi | \iota_\xi{\threesecofA}\big\rangle$;
		
	\item[$\mathrm{(III)}$]
	For all $x\in\secA$ and $\xi\in\secAd$, the maps $L_{x\circ\xi}-\CBr{L_x}{L_\xi} \colon \secAd\rightarrow\secAd$ and $L_{\xi\circ x}-\CBr{L_\xi}{L_x}:\secA\rightarrow\secA$ are both $\CinfM$-linear, and their traces coincide,
\[
\trace\big(L_{\xi\circ x}-\CBr{L_\xi}{L_x}\big)=\trace\big(L_{x\circ\xi}-\CBr{L_x}{L_\xi}\big)=2\big\langle\dd(\xi) | {\dstar}(x)\big\rangle
		-2\big\langle \iota_x\oldpsi | \iota_\xi{\threesecofA}\big\rangle .
\]
Here $x \circ \xi := -\iota_\xi (d_\ast x) + L_x \xi $ and $\xi \circ x := L_\xi x -\iota_x(\dd \xi)$.
\item[$\mathrm{(IV)}$]
		For all $f\in \CinfM$, $x\in\secA$, and $\xi\in\secAd$,
		\begin{eqnarray*}
			\Delta(f)&=&\Delta_\ast(f) =\frac{1}{2}(L_{X_0}+L_{\xi_0})(f),\\
			\Delta(x)&=&(\frac{1}{2}L_{X_0}+\frac{1}{2}L_{\xi_0} +Q_1)(x),\\
			\Delta_\ast(\xi) &=&(\frac{1}{2}L_{X_0}+\frac{1}{2}L_{\xi_0} +Q_3)(\xi).
			\end{eqnarray*}
	\end{itemize}
\end{prop}
\begin{proof}
The proof consists of four steps: $\mathrm{(I)} \Leftrightarrow \mathrm{(II)}$, $ \mathrm{(I)} \Rightarrow \mathrm{(III)}$,   $\mathrm{(III)}\Rightarrow \mathrm{(IV)}$, and $ \mathrm{(IV)} \Rightarrow \mathrm{(I)} $.	
\begin{itemize}
\item $\mathrm{(I)}~\Leftrightarrow~\mathrm{(II)}$:
We claim that
\begin{eqnarray}\label{466}
	L_{{\dstar}(f)}+L_{\dd(f)}=0 \quad\mbox{as a map}\quad
	\wedgeAdegree{\bullet}\rightarrow\wedgeAdegree{\bullet}.
\end{eqnarray}
{In fact, for all $f\in \CinfM$, $x\in\secA$, and $\xi\in\secAd$, we have, on the one hand,
\begin{eqnarray*}
\Delta_\ast\big(\langle x|f\xi\rangle\big)
&=&\Delta_\ast\big(f\langle x|\xi\rangle\big) \equalbyreason{\eqref{41}} \Delta_\ast(f)\langle x|\xi\rangle +f\Delta_\ast\big(\langle x|\xi\rangle\big) -\BrAd{\dd(f)}{\langle x|\xi\rangle}
+\BrAd{f}{\dd\big(\langle x|\xi\rangle\big)}\\
&=&\Delta_\ast(f)\langle x|\xi\rangle +f\Delta_\ast\big(\langle x|\xi\rangle\big)
-\big(L_{{\dstar}(f)}+L_{\dd(f)}\big)\langle x|\xi\rangle;
\end{eqnarray*}
On the other hand, we have
\begin{align*}
\Delta_\ast\big(\langle x|f\xi\rangle\big) &\equalbyreason{$\mathrm{(I)}$}
\big\langle \Delta(x)|f\xi\big\rangle +\big\langle x|\Delta_\ast(f\xi)\big\rangle
-2f\big\langle \iota_x\oldpsi | \iota_\xi{\threesecofA}\big\rangle\\
&\equalbyreason{\eqref{411}}
f\big\langle \Delta(x)|\xi\big\rangle
+\big\langle x|f\Delta_\ast(\xi)+\Delta_\ast(f)\,\xi
+\dd(\BrAd{f}{\xi})-\BrAd{\dd(f)}{\xi}+\BrAd{f}{\dd(\xi)}\big\rangle\\
&\qquad\qquad \qquad \qquad -2f\big\langle \iota_x\oldpsi | \iota_\xi{\threesecofA}\big\rangle\\
&= f \big\langle x|\Delta_\ast(\xi)\big\rangle
+f \big\langle \Delta(x)|\xi\big\rangle +\Delta_\ast(f) \langle x|\xi\rangle
-\big\langle x|\big(L_{{\dstar}(f)}+L_{\dd(f)}\big)\xi\big\rangle
-2f\big\langle \iota_x\oldpsi | \iota_\xi{\threesecofA}\big\rangle\\
&\equalbyreason{$\mathrm{(I)}$}
f \Delta_\ast\big(\langle x|\xi\rangle\big)
+\Delta_\ast(f) \langle x|\xi\rangle
-\big\langle x|\big(L_{{\dstar}(f)}+L_{\dd(f)}\big)\xi\big\rangle.
\end{align*}
Here we have used the identity
\[
 \big\langle x| {\dd}\big(\BrAd{f}{\xi}\big)-\BrAd{\dd(f)}{\xi} +\BrAd{f}{\dd(\xi)}\big\rangle = \big(\CBr{L_\xi}{L_x}-L_{\xi\circ x}\big)(f).
\]
Thus, we obtain $(L_{{\dstar}(f)}+L_{\dd(f)})\langle x|\xi\rangle
=\big\langle x|\big(L_{{\dstar}(f)}+L_{\dd(f)}\big)\xi\big\rangle$, which implies Equation~\eqref{466}.}

Note that
\begin{eqnarray*}
(\Delta f)\Omega= (\partial {\dstar} f)\Omega \equalbyreason{\eqref{LLL1}} -L_{{\dstar}(f)}\Omega
\equalbyreason{\eqref{466}} L_{\dd (f)}\Omega \equalbyreason{\eqref{LLL3}}
(\partial_\ast \dd f)\Omega = (\Delta_\ast f)\Omega,
\end{eqnarray*}
which implies that $\Delta f=\Delta_*f$, for all $f\in \CinfM$. Thus, we have the equivalence between $\mathrm{(I)}$ and $\mathrm{(II)}$.
\item  $\mathrm{(I)} \Rightarrow \mathrm{(III)}$:  By Lemma~\ref{42} and Equation~\eqref{466},
			\begin{eqnarray*}
				(L_{x\circ\xi}-\CBr{L_x}{L_\xi})(f)=0, \qquad\forall f\in \CinfM.
			\end{eqnarray*}
Thus the map
\[
L_{x\circ\xi}-\CBr{L_x}{L_\xi} \colon \secAd\rightarrow\secAd
\]
is $\CinfM$-linear.
By~\cite{CS}*{Proposition 4.7}, we have
\[
\big(L_{x\circ\xi}-\CBr{L_x}{L_\xi}\big)\Omega\otimes V
=\big(2\big\langle {\dstar}(x)|\dd(\xi)\big\rangle -\big\langle \Delta(x)|\xi\big\rangle
-\big\langle x|\Delta_\ast(\xi)\big\rangle  +\Delta_\ast\big(\langle x|\xi\rangle\big)
\big)\Omega\otimes V
\]
Using relation $\mathrm{(I)}$, we obtain
\[
\trace\big(L_{x\circ \xi}-\CBr{L_x}{L_\xi}\big) =  2\big\langle{\dstar}(x) | \dd(\xi)\big\rangle-2\big\langle \iota_x\oldpsi | \iota_\xi{\threesecofA}\big\rangle.
\]	
Similarly, it follows from $\mathrm{(II)}$ that
\[
\trace\big(L_{\xi\circ x}-\CBr{L_\xi}{L_x}\big) =  2\big\langle{\dstar}(x) | \dd(\xi)\big\rangle-2\big\langle \iota_x\oldpsi | \iota_\xi{\threesecofA}\big\rangle.
\]		
\item $\mathrm{(III)}\Rightarrow\mathrm{(IV)}$:
We need the following relations proved in \cite{CS}*{Propositions 4.10 and 4.11}.
\begin{eqnarray}
\big(2\Delta(f)-(L_{X_0}+L_{\xi_0})(f)\big)\Omega\otimes s\otimes V
&=&\Omega\otimes \Big(s\otimes \big(L_{\dd(f)}+L_{{\dstar}(f)}\big)V \nonumber\\&&\qquad-\big(L_{\dd(f)} + L_{{\dstar}(f)} \big) s\otimes V\Big), \label{(47)}
\\
\big\langle 2\Delta(x)-(L_{X_0}+L_{\xi_0})(x)|\xi\big\rangle\Omega\otimes s\otimes V
	&=&2\big\langle {\dstar}(x)|\dd(\xi)\big\rangle\Omega\otimes s\otimes V
\nonumber\\&&\qquad+\big(\CBr{L_x}{L_\xi}-L_{x\circ\xi}\big)\Omega\otimes s\otimes V
\nonumber\\&&\qquad\qquad+\Omega\otimes \big(\CBr{L_x}{L_\xi}-L_{x\circ\xi}\big)s\otimes V.
\label{Pro 4.11}
\end{eqnarray}
for all $x\in\secA, \xi\in\secAd$ and $f\in \CinfM$, where $X_0$ and $\xi_0$ are modular elements defined in Equations~\eqref{X_0} and~\eqref{xi_0}.

We start by proving that
	\begin{eqnarray*}
					\Delta(f)=\frac{1}{2}(L_{X_0}+L_{\xi_0})(f), \qquad\forall f\in \CinfM.
	\end{eqnarray*}
{In fact, the assumption that $L_{x\circ\xi}-\CBr{L_x}{L_\xi}$ is $\CinfM$-linear implies that
\begin{eqnarray*}
(L_{x\circ\xi}-\CBr{L_x}{L_\xi})f=0, \quad \forall f\in\CinfM.
\end{eqnarray*}
By Lemma~\ref{42}, we also have
\begin{eqnarray*}
L_{{\dstar}(f)}+L_{\dd(f)}=0 \quad\mbox{as a map}\quad \wedgeAdegree{\bullet}\rightarrow\wedgeAdegree{\bullet},
\end{eqnarray*}
which further implies that
\begin{eqnarray*}
(L_{\dd(f)}+L_{{\dstar}(f)})s &=&(L_{a_*(\dd f)+a_A({\dstar} f)}g)\,\dM q^1\wedge \ldots\wedge \dM q^n\\
&&\qquad+g\sum_{i=1}^n d q^1\wedge\ldots\wedge d(L_{a_\ast(\dd f)+a_A({\dstar} f)}q^i)
\wedge\ldots\wedge d q^n\\ &=&0,
\end{eqnarray*}
where  $s=gd q^1\wedge \ldots\wedge d q^n$ is a local coordinates expression of $s$.
Thus, the right-hand side of Equation \eqref{(47)} vanishes, which implies that
\[
	\Delta(f)=\frac{1}{2}(L_{X_0}+L_{\xi_0})(f).
\]
}

To obtain the second equation in $\mathrm{(IV)}$,
we use the assumption $\mathrm{(III)}$ and Equation \eqref{Pro 4.11} to obtain
\[
\big\langle \Delta(x)-\frac{1}{2}(L_{X_0}+L_{\xi_0})(x)|\xi\rangle =\langle \iota_x\oldpsi | \iota_\xi{\threesecofA}\big\rangle =\iota_{\iota_x\oldpsi}\iota_\xi{\threesecofA}
	=\iota_\xi \iota_{\iota_x\oldpsi}{\threesecofA}
	=\langle \iota_{\iota_x\oldpsi}{\threesecofA} | \xi\rangle,
\]
which implies that $\Delta(x)=\frac{1}{2}(L_{X_0}+L_{\xi_0})(x) +\iota_{\iota_x\oldpsi} {\threesecofA}$.
If we assume that $\oldpsi=\oldpsi_1\wedge \oldpsi_2\wedge \oldpsi_3$, then the term
\[
\iota_{\iota_x\oldpsi}{\threesecofA}	 = \big({\threesecofA}(\oldpsi_2,\oldpsi_3)\iota_{\oldpsi_1}
		- {\threesecofA}(\oldpsi_1,\oldpsi_3)\iota_{\oldpsi_2}
		+ {\threesecofA}(\oldpsi_1,\oldpsi_2)\iota_{\oldpsi_3}\big)x.
\]
Thus, we obtain
	\begin{eqnarray*}
			\Delta(x)&=&\big(\frac{1}{2}L_{X_0}+\frac{1}{2}L_{\xi_0}
					+{\threesecofA}(\oldpsi_2,\oldpsi_3)\wedge \iota_{\oldpsi_1}
					-{\threesecofA}(\oldpsi_1,\oldpsi_3)\wedge \iota_{\oldpsi_2}
					+{\threesecofA}(\oldpsi_1,\oldpsi_2)\wedge \iota_{\oldpsi_3}
					\big)(x)\\
					&=&(\frac{1}{2}L_{X_0}+\frac{1}{2}L_{\xi_0}
					+ Q_1)(x),
	\end{eqnarray*}
as desired. The proof for the third equation is similar and thus omitted.
\item $\mathrm{(IV)}\Rightarrow\mathrm{(I)}$: This implication follows directly from Lemma~\ref{Q1Q3}.
\end{itemize}
\end{proof}

\begin{prop}\label{Main-pre-Thm-3}
Assume that $(A,\BrA{\tobefilledin}{\tobefilledin},\BrAd{\tobefilledin}{\tobefilledin}, a_A, a_\ast, {\threesecofA}, \oldpsi)$ is a proto-bialgebroid. Then we have for all $k \geq 0$,
\begin{itemize}
\item[$\mathrm{(V)}$]
$\Delta=\frac{1}{2}(L_{X_0}+L_{\xi_0}) +Q_1+Q_2 \colon \wedgeAdegree{k} \rightarrow \wedgeAdegree{k}$,
\item[$\mathrm{(VI)}$]
$\Delta_\ast =\frac{1}{2}(L_{X_0}+L_{\xi_0})+Q_3+Q_4 \colon \OmegaAdegree{k} \rightarrow \OmegaAdegree{k}$.
\end{itemize}
\end{prop}
\begin{proof}
First of all, we note that the following two assertions hold:
	\begin{itemize}
		\item[$\mathrm{(VII)}$]
		${\dstar}\big(\BrA{r_1}{r_2}\big) =\BrA{{\dstar}(r_1)}{r_2} +(-1)^{\puredegree{r_1}-1} \BrA{r_1}{{\dstar}(r_2)} +\iota_\oldpsi(\threesecofA \wedge r_1 \wedge r_2)$;
		\item[$\mathrm{(VIII)}$]
		$\Delta(r_1\wedge r_2)	=(\Delta r_1)\wedge r_2 +r_1\wedge(\Delta r_2)
		+(-1)^{\puredegree{r_1}}\iota_\oldpsi(\threesecofA \wedge r_1 \wedge r_2)$,
	\end{itemize}
	for all $r_1,r_2\in\wedgeA$, $\omega_1,\omega_2\in\OmegaA$.
In fact,  $\mathrm{(VII)}$ is a generalized form of axiom $(3)$ in Definition \ref{Def:Proto-bialgebroids}.  By  Equation~\eqref{prop 4.5}, statements $\mathrm{(VII)}$ and $\mathrm{(VIII)}$ are equivalent.
	
Note that relations $\mathrm{(V)}$ and $\mathrm{(VI)}$ are symmetric. It suffices to prove $\mathrm{(V)}$.
We firstly show that the relation $\mathrm{(III)}$ holds.
In fact, by Lemma~\ref{42},  one has
\[
\big(L_{x\circ\xi}-\CBr{L_x}{L_\xi}\big)(f\eta) =f\big(L_{x\circ\xi}-\CBr{L_x}{L_\xi}\big)(\eta)
	-\big\langle {\dstar} \BrA{f}{x}  -\BrA{{\dstar} f}{x} +\BrA{f}{{\dstar}(x)}|\xi\big\rangle\eta,
\]
for all $x\in\secA,\xi,\eta\in\secAd,f\in \CinfM$.
It follows from $\mathrm{(VII)}$ that $L_{x\circ\xi}-\CBr{L_x}{L_\xi}$ is a $\CinfM$-linear endomorphism of $\OmegaA$.
			
Let $\{\xi^1,\ldots,\xi^n\}$ be a local basis of $\secA$ and $\{\theta_1,\ldots,\theta_n\}$ be the dual basis of $\secAd$. It follows that
	\begin{eqnarray*}
		\trace\big(L_{x\circ\xi}-\CBr{L_x}{L_\xi}\big) &=&\big\langle \xi^a|
				\big(L_{x\circ\xi}-\CBr{L_x}{L_\xi}\big)(\theta_a)\big\rangle\\
			&=& \big\langle \iota_{\theta_a}{\dstar}(x)|\iota_{\xi^a}\dd(\xi)\big\rangle
				+\oldpsi\big(x,{\threesecofA}(\xi,\theta_a),\xi^a\big)
				\qquad\mbox{by Lemma~\ref{Lemma 5.1}}\\
			&=& 2\big\langle {\dstar}(x)|\dd(\xi)\big\rangle - 2\big\langle \iota_x\oldpsi | \iota_\xi{\threesecofA}\big\rangle.
	\end{eqnarray*}
Here the last equality holds since
\begin{eqnarray*}
\oldpsi\big(x,{\threesecofA}(\xi,\theta_a),\xi^a\big) =-\oldpsi\big(x, \xi^a, {\threesecofA}(\xi,\theta_a)\big) =-\big\langle \iota_{\xi^a}\iota_x\oldpsi | \iota_{\theta_a} \iota_\xi{\threesecofA}\big\rangle = -2\big\langle \iota_x\oldpsi | \iota_\xi{\threesecofA}\big \rangle.
\end{eqnarray*}
In a similar manner, we see that $L_{\xi\circ x}-\CBr{L_\xi}{L_x} $ is $\CinfM$-linear and
\[
\trace\big(L_{\xi\circ x}-\CBr{L_\xi}{L_x}\big) =2\big\langle\dd(\xi) | {\dstar}(x)\big\rangle
-2\big\langle \iota_x\oldpsi | \iota_\xi{\threesecofA}\big\rangle .
\]
Thus, the relation $\mathrm{(III)}$ holds.
By Proposition \ref{Main-pre-thm-2}, we see that relations in $\mathrm{(IV)}$ hold, which mean that
\[
\Delta=\frac{1}{2}(L_{X_0}+L_{\xi_0}) +Q_1+Q_2 \colon \wedgeAdegree{0} \rightarrow \wedgeAdegree{0},
\]
and
\[
\Delta=\frac{1}{2}(L_{X_0}+L_{\xi_0}) +Q_1+Q_2 \colon \wedgeAdegree{1} \rightarrow \wedgeAdegree{1}.
\]
We now check that this relation holds as a map $\wedgeAdegree{2}\rightarrow\wedgeAdegree{2}$, i.e.,  for all $x,y\in \secA$, we have
\begin{equation}\label{Eqt:temp3-23}
\Delta(x\wedge y)  = (\frac{1}{2}L_{X_0} +\frac{1}{2}L_{\xi_0} +Q_1+Q_2)(x\wedge y).				
\end{equation}
In fact, we have
\begin{eqnarray*}
	\Delta(x\wedge y) &\equalbyreason{ \mbox{$\mathrm{(VIII)}$} }&(\Delta x)\wedge y
			+x\wedge(\Delta y) -{\threesecofA}\big(\oldpsi(x,y)\big)\\
			&\equalbyreason{ \mbox{$\mathrm{(IV)}$} }& \big(\frac{1}{2}(L_{X_0}+L_{\xi_0})(x)
			+\iota_{\iota_x\oldpsi}{\threesecofA}\big)\wedge y
+x\wedge\big(\frac{1}{2}(L_{X_0}+L_{\xi_0})(y)
			+\iota_{\iota_y\oldpsi}{\threesecofA}\big)
			-\iota_{\iota_y\iota_x\oldpsi}{\threesecofA}\\
			&=&\frac{1}{2}(L_{X_0}+L_{\xi_0})(x\wedge y)
			+(\iota_{\iota_x\oldpsi}{\threesecofA})\wedge y
			+x\wedge(\iota_{\iota_y\oldpsi}{\threesecofA})
			-\iota_{\iota_y\iota_x\oldpsi}{\threesecofA} \\
			&=&\big(\frac{1}{2}(L_{X_0}+L_{\xi_0})
			+Q_1+Q_2\big)(x\wedge y).
\end{eqnarray*}
This proves the claim~\eqref{Eqt:temp3-23}.
The cases for $k\geqslant 3$	follow from a standard induction argument and are omitted.
 \end{proof}

\subsection{Proofs of theorems}
\subsubsection{Proof of Theorem~\ref{MAIN THM}}
First of all, we need to compute the square of $\bigD$.
For all $x\in\secA, \xi\in\secAd, v\in\wedgeAdegree{\bullet}$ and $\oldpsi\in\OmegaAdegree{3}$, one has
\[
({\threesecofA}\wedge\iota_{\oldpsi} +\iota_{\oldpsi}\circ{\threesecofA})(v) = \langle{\threesecofA}|\oldpsi\rangle\, v -Q_1(v)-Q_2(v),
\]
and
\[
	\partial\circ\iota_{\oldpsi}
	+\iota_{\oldpsi}\circ\partial
	=\iota_{\dd(\oldpsi)}.
\]
Meanwhile, a simple computation that adapted from \cite{CS} yields that
\begin{align*}
\breve{{\dstar}}(r\otimes l) &=({\dstar}(r)+\frac{1}{2}X_0\wedge r)\otimes l, &
\breve{\partial}(r\otimes l) &=(-\partial(r)+\frac{1}{2}\iota_{\xi_0}r)\otimes l,
\end{align*}
for all $r\in\wedgeA$ and $l \in\GammaL$.
Thus, the operator $\bigD$ is related to the modular elements $X_0$ and $\xi_0$ by
\begin{eqnarray*}
\bigD=\breve{{\dstar}}+\breve{\partial}+{\threesecofA}-\iota_{\oldpsi}
={\dstar}-\partial+\frac{1}{2}(X_0\wedge+\iota_{\xi_0})+{\threesecofA}-\iota_{\oldpsi}.
\end{eqnarray*}
Combining these equations with the BV identity~\eqref{BV} for $\partial$,  one has  for all $v\in\wedgeAdegree{\bullet}$,
\begin{align*}
	\bigD^2(v)
	=&\big({\dstar^2}+L_{\threesecofA}
	+\frac{1}{2}{\dstar}(X_0)
	+\frac{1}{2}\iota_{\xi_0}{\threesecofA}
	-\partial({\threesecofA})\big)(v)
	& (\in \wedgeAdegree {\bullet+2})
	\\
	&\quad+\big(-{\dstar}\circ\iota_{\oldpsi}
	-\iota_{\oldpsi}\circ{\dstar}
	+\partial^2
	-\frac{1}{2}\iota_{\iota_{X_0}\oldpsi}
	-\frac{1}{2}\iota_{\dd(\xi_0)}\big)(v)
	& (\in \wedgeAdegree {\bullet-2})
	\\
	&\quad\quad+{\dstar}({\threesecofA})\wedge v
	& (\in \wedgeAdegree {\bullet+4})
	\\
	&\quad\quad\quad+\iota_{\dd(\oldpsi)}v
	& (\in \wedgeAdegree {\bullet-4})
	\\
	&\quad\quad\quad\quad+\big(\frac{1}{2}L_{X_0}
	+\frac{1}{2}L_{\xi_0}
	-\Delta
	+Q_1+Q_2\big)(v)
	& (\in \wedgeAdegree {\bullet})	
\\
	&\quad\quad\quad\quad\quad+\big(\frac{1}{4}\langle\xi_0|X_0\rangle
	-\frac{1}{2}\partial(X_0)-\langle{\threesecofA}|\oldpsi\rangle\big) v
	\qquad\qquad\qquad\qquad\qquad & (\in \CinfM\, v).
\end{align*}
\begin{equation}\label{Eqt:Dsquarecomputed}
\end{equation}
Note that the implication $(3)\Rightarrow (2)$ follows from the definition of Dirac generating operators.
The rest proof of Theorem \ref{MAIN THM} is divided into the following three steps.

\smallskip
\textbf{Step 1}: $(1)\Rightarrow (2)$ Namely, {$\bigD^2$ is smooth function on $M$, i.e., $\bigD^2\in \CinfM$ under the hypothesis that $(A,\BrA{\tobefilledin}{\tobefilledin}, \BrAd{\tobefilledin}{\tobefilledin}, a_A, a_\ast, {\threesecofA}, \oldpsi)$ is a proto-bialgebroid.
	
Recall that the Jacobiator of the bracket $[\cdot,\cdot]_\ast$ on $\Gamma(A^\ast)$ is controlled by the element $\tau$ and the derivation $d_A$, i.e.,
\[
{\dstar^2}+L_{\threesecofA}=0,
\]
where $L_\tau=\BrA{\threesecofA}{\cdot}$.
Combining with Equation~\eqref{Main-prop-1-1}, the first line on the right hand side of Equation \eqref{Eqt:Dsquarecomputed} vanishes;
by Equation \eqref{Eqt:Main-prop-2}, the second line vanishes;
by ${\dstar}({\threesecofA})=0$ and $\dd(\oldpsi)=0$, the third and fourth lines are zero; by $\mathrm{(V)}$ of Proposition \ref{Main-pre-Thm-3}, the fifth line is also zero.
In summary, the above terms combine to yield
\begin{eqnarray*}
\bigD^2(v) =\breve{f} v,
\end{eqnarray*}
where
\begin{equation*}
\breve{f} =\frac{1}{4}\langle\xi_0|X_0\rangle -\frac{1}{2}\partial(X_0) -\langle{\threesecofA}|\oldpsi\rangle,
\end{equation*}
which is exactly the characteristic function.
}

\smallskip \textbf{Step 2}: $(2)\Rightarrow (1)$
Suppose that there exists a smooth function $g \in \CinfM$ such that $\bigD^2(v)=g v$ for all $v\in\wedgeAdegree{\bullet}$.
To see that the septuple  $(A$, $\BrA{\tobefilledin}{\tobefilledin}$, $\BrAd{\tobefilledin}{\tobefilledin}$, $a_A$, $a_\ast$, ${\threesecofA}$, $\oldpsi)$ is a proto-bialgebroid, we need to check axioms in Definition \ref{Def:Proto-bialgebroids}.

Consider the constant function $v=1$ in Equation~\eqref{Eqt:Dsquarecomputed}.
We have
\begin{align}
\dd(\oldpsi) &=0, & {\dstar}({\threesecofA})&= 0, \nonumber\\
 \frac{1}{2}{\dstar}(X_0)+\frac{1}{2}\iota_{\xi_0}{\threesecofA}-\partial({\threesecofA}) &= 0, &
g&= \breve{f}, \notag \\
\label{AAA1}
{\dstar^2}+L_{\threesecofA} &=0, &  \frac{1}{2}L_{X_0} +\frac{1}{2}L_{\xi_0} -\Delta +Q_1 +Q_2 &= 0,
\end{align}
\begin{equation}\label{AAA2}
-{\dstar}\circ\iota_{\oldpsi} -\iota_{\oldpsi}\circ{\dstar} +\partial^2 -\frac{1}{2}\iota_{\iota_{X_0}\oldpsi} -\frac{1}{2}\iota_{\dd(\xi_0)} =0.
 \end{equation}

Axioms $(4)$ and $(5)$ are exactly the two equations in the first row.
The Equations in \eqref{AAA1} imply that axioms $(2)$ and $(3)$ are fulfilled, respectively.
By Lemma \ref{Lem:temp327} and Equation \eqref{AAA2}, we have
\[
\dd^2 +L_{\oldpsi} -\partial_\ast(\oldpsi) +\frac{1}{2}\iota_{X_0}\oldpsi +\frac{1}{2}\dd(\xi_0)
=0.
\]
Applying to the constant function $1$, we get $-\partial_\ast(\oldpsi) + \frac{1}{2}\iota_{X_0}\oldpsi +\frac{1}{2}\dd(\xi_0)=0.$
It follows that $\dd^2+L_{\oldpsi}=0$, which implies that axiom $(1)$ is also satisfied.
\smallskip

\textbf{Step 3}: $(1)\Rightarrow (3)$
We need to verify the operator $\bigD$ in~\eqref{expression of D} satisfies conditions $(a)\sim (c)$ in Definition \ref{Def:DGO}.
As condition $(c)$ follows from Step 1,
we only need to check conditions  $(a)$ and $(b)$, i.e.,
\begin{align}\label{Eq: a}
	\CBr{\bigD}{f} &\in \Gamma(A\oplus A^\ast), \\\label{Eq: b}
\CBr{\CBr{\bigD}{x+\xi}}{y+\eta} &\in\Gamma(A\oplus A^\ast),
\end{align}
for all $f\in\CinfM$, $x,y\in\secA$, and $\xi,\eta\in\secAd$,.
Firstly, a direct calculation shows that
\begin{eqnarray}\label{Tech-1}
~\CBr{{\threesecofA}-\iota_{\oldpsi}}{f}(v) &=&0,\\
\label{Tech-2}
\mbox{ and }~\CBr{\CBr{{\threesecofA}-\iota_{\oldpsi}}{x+\xi}}{y+\eta}(v) &=&-(\iota_\eta\iota_\xi{\threesecofA})\wedge v -\iota_{(\iota_y\iota_x\oldpsi)}v,
\end{eqnarray}
for all $x,y\in\secA$, $\xi,\eta\in\secAd$, $f\in \CinfM$, and $v\in\wedgeAdegree{\bullet}$.

For all $v\in\wedgeAdegree{\bullet}$, we have
\begin{eqnarray*}
	\CBr{\bigD}{f}(v) &\equalbyreason{\eqref{expression of D}} &
	\CBr{{\dstar}-\partial+\frac{1}{2}(X_0+\iota_{\xi_0})+{\threesecofA}-\iota_{\oldpsi}}{f}(v)\\
	&\equalbyreason{\eqref{Tech-1}}&
	\CBr{{\dstar}-\partial}{f}(v) =\big({\dstar}(f)+\dd(f)\big)\cdot v
	 \in \Gamma(A\oplus A^\ast)\cdot v.
\end{eqnarray*}
This proves~\eqref{Eq: a}.

For Equation~\eqref{Eq: b}, we compute, for all $x+\xi,y+\eta\in\Gamma(A\oplus A^*)$, $v\in\wedgeAdegree{\bullet}$,
\begin{eqnarray*}
\CBr{\CBr{\bigD}{x+\xi}}{y+\eta}(v) &\equalbyreason{\eqref{expression of D}} &
\CBr{\CBr{{\dstar}-\partial+\frac{1}{2}(X_0+\iota_{\xi_0})+{\threesecofA}-\iota_{\oldpsi}}{x+\xi}}{y+\eta}(v)\\
&\equalbyreason{\eqref{Tech-2}}&
\CBr{{\dstar}(x) +L_\xi +L_x-\partial(x) -\iota_{\dd(\xi)}}{y+\eta}(v)
-(\iota_\eta\iota_\xi{\threesecofA})\wedge v -\iota_{(\iota_y\iota_x\oldpsi)}v\\
&=&\big(\BrA{x}{y} +L_\xi y -\iota_\eta\big({\dstar}(x)\big) -\iota_\eta\iota_{\xi}{\threesecofA}\big)\wedge v \qquad\in\secA\wedge v \\
&&\qquad+\big(\BrAd{\xi}{\eta} +L_x\eta -\iota_y\big(\dd(\xi)\big) -\iota_y\iota_{x}\oldpsi \big)(v)
\qquad\in\secAd \cdot v \\
&\equalbyreason{\eqref{Eqn:Dorfmanbracket}}&
\big((x+\xi)\circ(y+\eta)\big)(v)\qquad\in\Gamma(A\oplus A^\ast)\cdot v.
\end{eqnarray*}

\subsubsection{Proof of Theorem~\ref{f-D2}}
In the proof of Theorem~\ref{MAIN THM}, we have obtained the expressions of the Dirac generating operator in~\eqref{expression of D} and the characteristic function in~\eqref{Eqt:brevefexplicitly}.
We now verify that the characteristic function
\[
\breve{f} =\frac{1}{4}\langle\xi_0|X_0\rangle -\frac{1}{2}\partial(X_0) - \langle{\threesecofA} | \oldpsi\rangle.
\]
is indeed an invariant of the proto-bialgebroid, i.e., it does not depend on the choices of the volume $s \in \Omega^m(M)$ of $M$, the top form $\Omega \in \Gamma(\wedge^n A^\ast)$ and its dual $V \in \Gamma(\wedge^n A)$.

Assume that $s^\prime = gs$ and that $\Omega^\prime = h\Omega, V^\prime = 1/h V$ for some nowhere vanishing smooth functions $g,h \in C^\infty(M)$.
Let $X_0^\prime \in \Gamma(A), \xi_0^\prime \in \Gamma(A^\ast)$ be the modular elements associated with $s^\prime, \Omega^\prime, V^\prime$.
Using Equations~\eqref{X_0}, we have for all $\xi \in \Gamma(A^\ast)$,
\begin{align*}
  \langle \xi | X^\prime_0 \rangle \Omega^\prime \otimes s^\prime &= L_\xi (\Omega^\prime \otimes s^\prime) = L_{\xi} (gh \Omega \otimes s) = a_\ast(\xi)(gh) (\Omega \otimes s) + gh L_{\xi}(\Omega \otimes s)\\
  &= \frac{a_\ast(\xi)(gh)}{gh}\Omega^\prime \otimes s^\prime + \langle \xi | X_0 \rangle \Omega^\prime \otimes s^\prime = \langle \xi | X_0 + d_\ast\log(gh) \rangle \Omega^\prime \otimes s^\prime.
\end{align*}
Thus, we have
\[
   X_0^\prime = X_0 + d_\ast (\log(gh)) = X_0 + d_\ast(\log g + \log h).
\]
By a similar computation and using~\eqref{xi_0}, we have
\[
 \xi_0^\prime = \xi_0 + d_A(\log g - \log h).
\]
Denote by $\partial^\prime$ the BV operator induced from $\Omega^\prime$ and $V^\prime$. Then we have for all $r \in \Gamma(\wedge^k A)$,
\begin{align*}
 \partial^\prime(r) &= -(-1)^{n(k+1)}(V^{\prime\sharp} \circ d_A \circ \Omega^{\prime\sharp})(r) =  -(-1)^{n(k+1)}1/h V^{\sharp} (d_A (h\Omega^{\sharp}(r))) \\
 &=  \partial(r) - (-1)^{n(k+1)}1/h V^{\sharp}(d_Ah \wedge \Omega^\sharp(r)) = \partial(r) - \iota_{d_A \log h} r.
\end{align*}
A direct computation shows that
\begin{align*}
  \breve{f}^\prime - \breve{f} &= \frac{1}{4}\left(\langle\xi^\prime_0|X^\prime_0\rangle - \langle \xi_0 | X_0 \rangle \right) - \frac{1}{2}(\partial^\prime(X^\prime_0) - \partial (X_0)) = 0.
  \end{align*}
Thus $\breve{f}$ is indeed an invariant of the proto-bialgebroid.

Finally, we check that $\breve{f}$ is completely determined by either statement $(a)$ or $(b)$.
Note that
\begin{eqnarray*}
L_{X_0}(\Omega\otimes s)\otimes V &=&(L_{X_0}\Omega)\otimes s\otimes V + \Omega\otimes (L_{X_0}s)\otimes V\\
&\equalbyreason{\eqref{LLL5}}& -\Omega\otimes s\otimes (L_{X_0}V) +\Omega\otimes (L_{X_0}s)\otimes V\\
&=& -2\Omega\otimes s\otimes (L_{X_0}V) +\Omega\otimes L_{X_0}(s\otimes V)\\
&\equalbyreason{\eqref{LLL2}\eqref{xi_0}}& \big(-2\partial(X_0)
+\langle \xi_0|X_0\rangle\big) \Omega\otimes s\otimes V\\
&\equalbyreason{\eqref{Eqt:brevefexplicitly}}& 4\big(\breve{f}+\langle{\threesecofA}|\oldpsi\rangle\big)
\Omega\otimes s\otimes V.
\end{eqnarray*}
The proves statement $(a)$.
The proof of Statement $(b)$ is similar and thus omitted.
\begin{Rem}
  Although the characteristic function $\breve{f} = \bigD^2$ is an invariant, the Dirac generating operator $\bigD$ does depend on the choices of $s, \Omega$ and $V$.
  In a sense, what curvatures are to connections, characteristic functions are to Dirac generating operators.
\end{Rem}

\section{Examples}\label{Sec:examples}
In this section, we consider some concrete proto-bialgebroids and compute the corresponding Dirac generating operators and characteristic functions.

\subsection{Proto-bialgebroids over Euclidean spaces}
We start with a proto-bialgebroid over an Euclidean space $\R^m$, which can be seen as a local model of proto-bialgebroid.

Let $(A,\BrA{\tobefilledin}{\tobefilledin},\BrAd{\tobefilledin}{\tobefilledin}, a_A, a_\ast, {\threesecofA}, \oldpsi)$ be a proto-bialgebroid, where $A= \R^m \times \R^n$ is a product vector bundle over $\R^m$ of rank $n$.
Choose a coordinate system $\{q^1,\ldots,q^m\}$ for base space $\R^m$. Suppose that  $\{e_1,\ldots,e_n\}$ is a basis of $\secA$ and that $\{e^1,\ldots,e^n\} $ is the dual basis of $\secAd$.
We assume that the two elements $\phi$ and $\tau$ are of the form
\begin{eqnarray*}
\oldpsi=\bar{\oldpsi}_{ijk}~ e^i\wedge e^j \wedge e^k
\in\OmegaAdegree{3}  \quad\mbox{ and }\quad
{\threesecofA}=\bar{{\threesecofA}}^{ijk}~ e_i \wedge e_j \wedge e_k
\in\wedgeAdegree{3}.
\end{eqnarray*}
And the skew-symmetric dull algebroid structures on $A$ and $A^\ast$ can be written as follow:
\begin{align*}
\BrA{e_i}{e_j} &=a_{ij}^k~ e_k, & a_A(e_i) &=A_i^\alpha \frac{\partial}{\partial q^\alpha},\\
\BrAd{e^i}{e^j} &=b_k^{ij}~ e^k, & a_\ast(e^i)&=B^{i\alpha}\frac{\partial}{\partial q^\alpha}.
\end{align*}
It follows that the two associated derivations $\dd$ and $\dstar$ are given by
\begin{align*}
\dstar(q^\alpha) &=B^{i \alpha} e_i, &
\dstar(e_i) &=-\frac{1}{2} b_i^{jk} e_j \wedge e_k,\\
\dd(q^\alpha) &=A_i^\alpha e^i, &
\dd(e^i) &= -\frac{1}{2} a^i_{jk} e^j \wedge e^k.
\end{align*}
Assume that
\begin{eqnarray*}
	\Omega&=&e^1\wedge \ldots \wedge e^n\in \OmegaAdegree{n},\\
	V&=&e_1\wedge \ldots \wedge e_n\in \wedgeAdegree{n},\\
	s&=&d q^1\wedge\ldots\wedge d q^m\in\Gamma(\wedge^m T^\ast \R^m).
\end{eqnarray*}
We now compute the associated BV operator $\partial$. It is clear that $\partial(q^\alpha)=0$.
And
\begin{align}\label{Eqn:ei}
\partial(e_i)
&= (-1)^{n-i}\partial\big(V^\sharp(e^1\wedge\ldots\wedge\widehat{e^i}\wedge\ldots\wedge e^n)\big) \notag \\
&\equalbyreason{\eqref{Eqn:V-d-partial}} (-1)^i V^\sharp\big(\dd (e^1\wedge\ldots\wedge\widehat{e^i}\wedge\ldots\wedge e^n)\big)
\notag \\
&=(-1)^i V^\sharp\big( \sum_{j=1}^{i-1}(-1)^{j-1} e^1\wedge\ldots\wedge e^{j-1}\wedge
 \dd(e^j)\wedge e^{j+1}\wedge\ldots\wedge\widehat{e^i}
\wedge\ldots\wedge e^n \notag\\
&\qquad+\sum_{k=i+1}^n(-1)^k  e^1\wedge\ldots\wedge\widehat{e^i}\wedge\ldots
\wedge e^{k-1}\wedge \dd(e^k)\wedge e^{k+1}\wedge\ldots\wedge e^n\big) \notag \\
&= (-1)^iV^\sharp\big(
\sum_{j=1}^{i-1}(-1)^{j-1} e^1\wedge\ldots\wedge
\big(-\frac{1}{2} a^j_{ij} e^i \wedge e^j
-\frac{1}{2} a^j_{ji} e^j \wedge e^i\big)
 \wedge \ldots\wedge\widehat{e^i}
\wedge\ldots\wedge e^n \notag \\
&\qquad+\sum_{k=i+1}^n(-1)^k e^1\wedge\ldots\wedge\widehat{e^i}\wedge\ldots\wedge
\big(-\frac{1}{2} a^k_{ik} e^i \wedge e^k
-\frac{1}{2} a^k_{ki} e^k \wedge e^i\big) \wedge \ldots\wedge e^n\big) \notag \\
&= (-1)^i V^\sharp\big((-1)^{j-1} (-1)^{i-j} a_{ji}^j \Omega
+(-1)^k (-1)^{k-i}a_{ik}^k \Omega\big) \notag \\
&=-\sum_{j=1}^{i-1} a_{ji}^j+\sum_{k=i+1}^n a_{ik}^k = a_{ij}^j.
\end{align}

Observe that
\begin{eqnarray*}
\langle e^i|X_0\rangle\Omega\otimes s
&\equalbyreason{\eqref{X_0}}&L_{e^i}(\Omega\otimes s)
=(L_{e^i}\Omega)\otimes s
+\Omega\otimes L_{a_*(e^i)}s\\
&=&\big(L_{e^i}(e^1\wedge\ldots\wedge e^n)\big)\otimes s
+\Omega\otimes L_{B^{i\alpha}\frac{\partial}{\partial q^\alpha}}
(\dM q^1\wedge\ldots\wedge\dM q^m)\\
&=&\sum_{j=1}^n e^1\wedge\ldots\wedge \BrAd{e^i}{e^j}\wedge\ldots\wedge e^n\otimes s\\
&&\qquad+\Omega\otimes \big(\sum_{\beta=1}^m\dM q^1\wedge\ldots
\wedge\dM(L_{B^{i\alpha}\frac{\partial}{\partial q^\alpha}}q^\beta)\wedge\ldots\wedge \dM q^m\big)\\
&=&(b_j^{ij}+\frac{\partial{B^{i\beta}}}{\partial q^\beta})\Omega\otimes s.
\end{eqnarray*}
Thus, we obtain
\begin{equation}
\label{Eqt:X0local}
X_0=(b_j^{ij}+\frac{\partial{B^{i\alpha}}}{\partial q^\alpha})e_i .
\end{equation}
In a similar approach, we have
\[
\xi_0=(\frac{\partial{A_i^{\alpha}}}{\partial q^\alpha}+a_{ij}^j)e^i .
\]
Using Equations~\eqref{Eqn:ei} and \eqref{Eqt:X0local}, we obtain
\begin{eqnarray*}
\partial(X_0)
=a_{ij}^j (b_k^{ik}+\frac{\partial{B^{i\alpha}}}{\partial q^\alpha})
- A_i^\beta\frac{\partial}{\partial q^\beta}(b_j^{ij}+\frac{\partial{B^{i\alpha}}}{\partial q^\alpha}).
\end{eqnarray*}

Hence, by Theorem \ref{f-D2},  the characteristic function in this case is given by
\begin{eqnarray*}
\bigD^2 &\equalbyreason{\eqref{Eqt:brevefexplicitly}}&
 \frac{1}{4}\langle\xi_0|X_0\rangle
-\frac{1}{2}\partial(X_0)
-\langle{\threesecofA}|\oldpsi\rangle\\
&=&\frac{1}{4}(\frac{\partial{A_i^{\alpha}}}{\partial q^\alpha}+a_{ij}^j)
(b_k^{ik}+\frac{\partial{B^{i\alpha}}}{\partial q^\alpha})
-\frac{1}{2}a_{ij}^j (b_k^{ik}+\frac{\partial{B^{i\alpha}}}{\partial q^\alpha})
+\frac{1}{2}A_i^\beta\frac{\partial}{\partial q^\beta}(b_j^{ij}+\frac{\partial{B^{i\alpha}}}{\partial q^\alpha})
-\bar{\oldpsi}_{ijk} \bar{{\threesecofA}}^{ijk}\\
&=&\frac{1}{4}(\frac{\partial{A_i^{\alpha}}}{\partial q^\alpha}-a_{ij}^j)
(b_k^{ik}+\frac{\partial{B^{i\alpha}}}{\partial q^\alpha})
+\frac{1}{2}A_i^\beta(\frac{\partial b_j^{ij}}{\partial q^\beta}
+\frac{\partial^2 B^{i \alpha}}{\partial q^\alpha \partial q^\beta})
-\bar{\oldpsi}_{ijk} \bar{{\threesecofA}}^{ijk}.
\end{eqnarray*}

\subsection{Proto-bialgebras}\label{Sec:ProtoLiebialgebra}
\subsubsection{$r$-dimensional proto-bialgebras}
We focus on proto-bialgebras $(\g,\g^\ast)$, i.e., proto-bialgebroids over a single point $M=\{\ast\}$. The corresponding Courant algebroid is denoted by $\g\oplus \g^\ast$, where $\g$ and $\g^\ast$ are $r$-dimensional vector spaces in duality.
\begin{Def}\label{Def:proto-bialgebra}
 A \textbf{proto-bialgebra} is a quadruple $(\g,\g^\ast,{\threesecofA},\oldpsi)$, where
 $(\g,\Brg{\tobefilledin}{\tobefilledin})$ and $(\g^\ast, \Brgd{\tobefilledin}{\tobefilledin})$ are two skew-symmetric dull algebras in duality, ${\threesecofA}  \in \wedge^3 \g  $, and $\oldpsi \in  \wedge^3 \g^\ast$, satisfying the following properties:
\begin{align}\label{Jacobi of g1}
\Brg{\Brg{x}{y}}{z}+\Brg{\Brg{y}{z}}{x}+\Brg{\Brg{z}{x}}{y}
		&=\iota_{\oldpsi}\big({\dstar} (x\wedge y\wedge z)\big), \\
\label{Jacobi of g^*}
\Brgd{\Brgd{\xi}{\eta}}{\chi} +\Brgd{\Brgd{\eta}{\chi}}{\xi} +\Brgd{\Brgd{\chi}{\xi}}{\eta}
		& = \iota_{{\threesecofA}}\big(\dd(\xi\wedge\eta\wedge \chi)\big), \\
\label{com:g-g}\Brg{{\dstar}(x)}{y}+\Brg{x}{{\dstar}(y)}+ \iota_{(\iota_y\iota_x\oldpsi)}{\threesecofA} &= {\dstar}\big(\Brg{x}{y}\big), \\
\label{dpsi-g}
		\dd(\oldpsi) &=0; \\
\label{dvarphi-g}
{\dstar}({\threesecofA}) &= 0,
	\end{align}
where the maps ${\dstar}\colon \wedge^{\bullet}\g \rightarrow \wedge^{\bullet+1}\g$ and $\dd \colon \wedge^{\bullet}\g^\ast \rightarrow \wedge^{\bullet+1}\g^\ast$  are the derivations arising from the skew-symmetric dull algebras $(\g^\ast, \Brgd{\tobefilledin}{\tobefilledin})$ and  $(\g,\Brg{\tobefilledin}{\tobefilledin})$, respectively.
\end{Def}
Let us choose a basis $\{e_1,e_2,\ldots,e_r\}$ of $\g$ with dual basis $\{e^1,e^2,\ldots,e^r\} $ of $\g^\ast$.  Assume that
\begin{eqnarray*}
\oldpsi=\bar{\oldpsi}_{ijk} e^i\wedge e^j \wedge e^k  \quad\mbox{ and }\quad
{\threesecofA}=\bar{{\threesecofA}}^{ijk} e_i \wedge e_j \wedge e_k.
\end{eqnarray*}
Let  $a_{ij}^k$ and $b_k^{ij}$ be the structure constants of the two brackets $[\cdot,\cdot]_\g$ and $[\cdot,\cdot]_{\g^\ast}$, that is,
\begin{eqnarray*}
\Brg{e_i}{e_j}=a_{ij}^k e_k  \quad\mbox{ and }\quad
\Brgd{e^i}{e^j}=b_k^{ij} e^k.
\end{eqnarray*}
Then the two operators ${\dstar}$ and $\partial$ are given by
\begin{eqnarray*}
\dstar(e_i)=-\frac{1}{2} b_i^{jk} e_j \wedge e_k, \qquad
\partial(e_i)=a_{ij}^j, \quad\mbox{and}\quad
\partial(e_i \wedge e_j) =-a_{ij}^k e_k +a_{ik}^k\, e_j -a_{jk}^k\, e_i.
\end{eqnarray*}
The Dorfman bracket of the corresponding Courant algebroid $\g\oplus \g^\ast$ reads
\[
(x+\xi)\circ(y+\eta) =(\Brg{x}{y} +\ad_\xi^\ast y -\ad^\ast_\eta x -\iota_\eta\iota_{\xi}{\threesecofA})
+(\Brgd{\xi}{\eta} +\ad^\ast_x\eta -\ad^\ast_y \xi -\iota_y\iota_{x}\oldpsi ),
\]
which is indeed \textit{skew-symmetric},  and satisfies the Jacobi identity.
Therefore, $(\g\oplus \g^\ast,\circ)$ is indeed a Lie algebra,  and $(\g\oplus \g^\ast,\circ,\langle \tobefilledin,\tobefilledin\rangle)$ is a \textit{quadratic Lie algebra}.

For expressions of modular elements $X_0\in \g$ and $\xi_0\in \g^\ast$,  we choose the particular volume forms $\Omega=  e^1\wedge \ldots \wedge e^r\in \wedge^{r}\g^\ast$ and $V=  e_1\wedge \ldots \wedge e_r\in \wedge^r \g$. Thus, we have
 \[
 \langle X_0|\xi\rangle=\trace(\ad_\xi), \qquad \langle \xi_0|x\rangle =\trace(\ad_x),\quad \forall \xi\in \g^\ast, x\in \g.
 \]
Hence, we have
\begin{eqnarray*}
	X_0=b_j^{ij} e_i \qquad\mbox{ and }\qquad \xi_0=a_{ij}^j e^i.
\end{eqnarray*}	
It follows that
\begin{eqnarray*}
\partial(X_0)=a_{ij}^j b_k^{ik}.
\end{eqnarray*}
The characteristic function $\bigD^2$, which is just a real number, is given by
\begin{eqnarray*}
\bigD^2 \equalbyreason{\eqref{Eqt:brevefexplicitly}} \frac{1}{4}\langle\xi_0|X_0\rangle -\frac{1}{2}\partial(X_0)
-\langle{\threesecofA}|\oldpsi\rangle =-\frac{1}{4} a_{ij}^j b_k^{ik} -\bar{\oldpsi}_{ijk} \bar{{\threesecofA}}^{ijk}.
\end{eqnarray*}

\subsubsection{$3$-dimensional case.}
Finally, we study $3$-dimensional proto-bialgebras in detail. In this case,
\begin{eqnarray}
\mbox{Equation}~\eqref{Jacobi of g1}&~\Leftrightarrow~&
\Brg{\Brg{x}{y}}{z}+\Brg{\Brg{y}{z}}{x}+\Brg{\Brg{z}{x}}{y}=0;\nonumber \\
\mbox{Equation}~\eqref{Jacobi of g^*}&~\Leftrightarrow~&
\Brgd{\Brgd{\xi}{\eta}}{\chi}
+\Brgd{\Brgd{\eta}{\chi}}{\xi}
+\Brgd{\Brgd{\chi}{\xi}}{\eta}=0;\nonumber \\
\mbox{Equation}~\eqref{com:g-g}&~\Leftrightarrow~&
\label{proto-detailcondition3}
{\dstar}(\Brg{x}{y})=\Brg{{\dstar}(x)}{y}+\Brg{x}{{\dstar}(y)}+\iota_{(\iota_{x\wedge y}\oldpsi)}{\threesecofA}.
\end{eqnarray}
Equations~\eqref{dpsi-g} and \eqref{dvarphi-g} hold by degree reasons.
Therefore, both $(\g=\mbox{span}\{e_1,e_2,e_3\},\Brg{\tobefilledin}{\tobefilledin})$ and $(\g^\ast =\mbox{span}\{e^1,e^2,e^3\},\Brgd{\tobefilledin}{\tobefilledin})$ are indeed Lie algebras.
However,  the pair $(\g,\g^\ast)$ is \textbf{not} a Lie bialgebra unless the term $\iota_{(\iota_{x\wedge y}\oldpsi)}{\threesecofA}$ in \eqref{proto-detailcondition3} vanishes, which holds if either the element ${\threesecofA}$ or $\oldpsi$ is zero. This special case has been well studied in \cite{HL}.

In what follows we assume that both ${\threesecofA}$ and $\oldpsi$ are \textbf{nontrivial}.
We assume that
\begin{eqnarray*}
		\oldpsi=\bar{\oldpsi} ~e^1\wedge e^2\wedge e^3 \in  \wedge^3\g^*  \quad
\mbox{ and }\quad
		{\threesecofA}=\bar{{\threesecofA}} ~e_1\wedge e_2\wedge e_3\in \wedge^3\g ,
	\end{eqnarray*}
for some nonzero numbers $\bar{\oldpsi}, \bar{{\threesecofA}}\in {\R}$.
Note that
 \begin{eqnarray*}
			{\dstar}(e_1)
				&=&-b_1^{12} e_1\wedge e_2
				-b_1^{23} e_2\wedge e_3
				-b_1^{31} e_3\wedge e_1,\\
				{\dstar}(e_2)
				&=&-b_2^{12} e_1\wedge e_2
				-b_2^{23} e_2\wedge e_3
				-b_2^{31} e_3\wedge e_1,\\
				{\dstar}(e_3)
				&=&-b_3^{12} e_1\wedge e_2
				-b_3^{23} e_2\wedge e_3
				-b_3^{31} e_3\wedge e_1.
	\end{eqnarray*}
 Substituting $(x,y)=(e_1,e_2), (e_2,e_3), (e_3,e_1)$ in Equation \eqref{proto-detailcondition3}, we obtain the following constraints on structure constants:
		\begin{equation}\label{detail-1}
			\begin{cases}
					a_{12}^3 b_3^{12}
					-a_{23}^1 b_1^{23}
					-a_{23}^2 b_1^{31}
					-a_{31}^1 b_2^{23}
					-a_{31}^2 b_2^{31}
					+\bar{\oldpsi} \bar{{\threesecofA}} =0  \\
					a_{12}^1 b_1^{23}
					+a_{12}^2  b_1^{31}
					+a_{12}^3 b_3^{23}
					+a_{12}^3 b_1^{12}
					+a_{23}^3 b_1^{23}
					+a_{31}^3 b_2^{23} =0  \\
					a_{12}^1 b_2^{23}
					+a_{12}^2 b_2^{31}
					+a_{12}^3 b_3^{31}
					+a_{12}^3 b_2^{12}
					+a_{23}^3 b_1^{31}
					+a_{31}^3 b_2^{31} =0;
			\end{cases}
		\end{equation}
		\begin{equation}\label{detail-2}
			\begin{cases}
					a_{12}^1 b_3^{12}
					+a_{23}^1 b_1^{12}
					+a_{23}^1 b_3^{23}
					+a_{23}^2 b_3^{31}
					+a_{23}^3 b_3^{12}
					+a_{31}^1 b_2^{12}
					=0
					\\
					-a_{12}^2 b_3^{31}
					-a_{12}^3 b_3^{12}
					+a_{23}^1 b_1^{23}
					-a_{31}^2 b_2^{31}
					-a_{31}^3 b_2^{12}
					+\bar{\oldpsi} \bar{{\threesecofA}}
					=0
					\\
					a_{12}^1 b_3^{31}
					+a_{23}^1 b_1^{31}
					+a_{23}^1 b_2^{23}
					+a_{23}^2 b_2^{31}
					+a_{23}^3 b_2^{12}
					+a_{31}^1 b_2^{31}
					=0;
			\end{cases}
		\end{equation}
		and		
\begin{eqnarray}\label{detail-3}
			\begin{cases}
					a_{12}^2 b_3^{12}
					+a_{23}^2 b_1^{12}
					+a_{31}^1 b_3^{23}
					+a_{31}^2 b_2^{12}
					+a_{31}^2 b_3^{31}
					+a_{31}^3 b_3^{12}
					=0
					\\
					a_{12}^2 b_3^{23}
					+a_{23}^2 b_1^{23}
					+a_{31}^1 b_1^{23}
					+a_{31}^2 b_1^{31}
					+a_{31}^2 b_2^{23}
					+a_{31}^3 b_1^{12}
					=0
					\\
					-a_{12}^1 b_3^{23}
					-a_{12}^3 b_3^{12}
					-a_{23}^1 b_1^{23}
					-a_{23}^3 b_1^{12}
					+a_{31}^2 b_2^{31}
					+\bar{\oldpsi} \bar{{\threesecofA}}
					=0.
			\end{cases}
\end{eqnarray} 		
 Moreover, the modular elements are
$X_0=\sum_{i,j=1}^3 b_j^{ij} e_i$ and $\xi_0=\sum_{i,j=1}^3 a_{ij}^j e^i$. We also have
\begin{eqnarray*}
\partial(X_0)
=(a_{12}^2+a_{13}^3)(b_2^{12}+b_3^{13})
+(a_{21}^1+a_{23}^3)(b_1^{21}+b_3^{23})
+(a_{31}^1+a_{32}^2)(b_1^{31}+b_2^{32}).
\end{eqnarray*}
Thus, the characteristic function  	 is given by the real number
\begin{eqnarray*}
\bigD^2
&=&\frac{1}{4}\langle\xi_0|X_0\rangle
-\frac{1}{2}\partial(X_0)
-\langle{\threesecofA}|\oldpsi\rangle\\
&=&-\frac{1}{4}(a_{12}^2+a_{13}^3)(b_2^{12}+b_3^{13})
-\frac{1}{4}(a_{21}^1+a_{23}^3)(b_1^{21}+b_3^{23})
-\frac{1}{4}(a_{31}^1+a_{32}^2)(b_1^{31}+b_2^{32})
-\bar{\oldpsi} \bar{{\threesecofA}}.
\end{eqnarray*}

\begin{Ex}
Consider $\g= \mathfrak{sl}(2;\R) = \mbox{span}\{e_1,e_2,e_3\}$ with the standard relations
\[
\Brg{e_1}{e_2}=2 e_2,\quad \Brg{e_1}{e_3}=-2 e_3,\quad \Brg{e_2}{e_3}=e_1.
\]
Suppose that $\g^\ast = \mbox{span}\{e^1, e^2, e^3\}$ is also endowed with a Lie bracket: $\Brgd{e^i}{e^j}=b_k^{ij} e^k, 1\leqslant i,j,k \leqslant 3$.
A well-known choice is given by	
\begin{eqnarray*}
		\Brgd{e^1}{e^2}=\frac{1}{4} e^2, \quad
		\Brgd{e^1}{e^3}=\frac{1}{4} e^3, \quad
		\Brgd{e^2}{e^3}=0,
	\end{eqnarray*}
which makes $(\g,\g^\ast)$ into a Lie bialgebra (see \cite{LuThesis}).

By Equations \eqref{detail-1}, \eqref{detail-2}, and \eqref{detail-3},  we see that $(\g,\g^\ast)$ is a proto-bialgebra if and only if there exists two real numbers $\bar{\phi}$ and $\bar{\tau}$  such that the following equations hold:
	\begin{eqnarray*}
		\begin{cases}
b_1^{2 3} (b_2^{1 2}-b_3^{3 1})=0, \\
b_1^{23} b_1^{3 1} =0, \\
b_1^{23} b_1^{1 2} =0,\\
b_1^{23} =b_2^{12}+b_3^{31} =\bar{\oldpsi} \bar{{\threesecofA}},\\
b_1^{31} + b_2^{23} =0, \\
b_1^{12} + b_3^{23} =0, \\
b_2^{31} =b_3^{12} =0.
		\end{cases}
	\end{eqnarray*}
Moreover, if $b_1^{23}=\bar{\oldpsi} \bar{{\threesecofA}}\neq 0$, then the above conditions reduce to
\begin{eqnarray*}
		\begin{cases}
     		b_1^{23} =2 b_2^{12} =2 b_3^{31} =\bar{\oldpsi} \bar{{\threesecofA}}, \\
			b_1^{12}=b_1^{31} =b_2^{23}=b_2^{31} =b_3^{12}=b_3^{23} =0.
		\end{cases}
	\end{eqnarray*}
For example, we can take $\bar{\oldpsi} \bar{{\threesecofA}}=b_1^{23}=1, b_2^{12}=b_3^{31}=\frac{1}{2}$, i.e.,
\[
\Brgd{e^1}{e^2}=\frac{1}{2} e^2,\quad \Brgd{e^1}{e^3}=-\frac{1}{2} e^3,\quad \Brgd{e^2}{e^3}=e^1.
\]
Since
		\begin{eqnarray*}
\xi_0&=&(a_{12}^2+a_{13}^3)e^1+(a_{21}^1+a_{23}^3) e^2+(a_{31}^1+a_{32}^2) e^3=0, \\
\mbox{ and } \quad
	X_0&=&(b_2^{12}+b_3^{13})e_1
		+(b_1^{21}+b_3^{23})e_2
		+(b_1^{31}+b_2^{32})e_3 =0,
		\end{eqnarray*}
it follows that the characteristic function reads
\begin{eqnarray*}
	\bigD^2=\breve{f}=-\bar{\oldpsi} \bar{{\threesecofA}}~(\neq 0).
\end{eqnarray*}

\end{Ex}

\end{document}